\documentclass[12pt]{amsart}
\usepackage{amsmath,mathtools,amsthm,amsfonts,bigints,amssymb, mathrsfs, tikz-cd, xcolor}

\newtheorem{theorem}{Theorem}[section]
\newtheorem{lemma}[theorem]{Lemma}

\newtheorem{corollary}[theorem]{Corollary}

\theoremstyle{definition}

\newtheorem{remark}[theorem]{Remark}

\newcommand{\R}{\mathbb R}%
\newcommand{\C}{\mathbb C}%
\newcommand{\N}{\mathbb N}%
\newcommand{\Sa}{\mathscr A}%

\newcommand{\z}{\mathfrak z}%
\newcommand{\mv}{\mathfrak v}%
\newcommand{\J}{\mathscr J}%
\newcommand{\F}{\mathscr F}%
\numberwithin{equation}{section}
\makeatletter

\renewcommand\subsubsection{\@secnumfont}{\bfseries}%
\renewcommand\subsubsection{\@startsection{subsubsection}{3}
  \z@{.5\linespacing\@plus.7\linespacing}{-.5em}%
  {\normalfont\bfseries}}

  \makeatother
\makeatletter
\@namedef{subjclassname@2020}{%
  \textup{2020} Mathematics Subject Classification}
\makeatother

\begin{document}

\title[Pointwise convergence along general curves]{Pointwise convergence of solutions of the Schr\"odinger equation along general curves on Damek-Ricci spaces}

\author[Utsav Dewan]{Utsav Dewan}
\address{Stat-Math Unit, Indian Statistical Institute, 203 B. T. Rd., Kolkata 700108, India}
\email{utsav97dewan@gmail.com \\
utsav\_r@isical.ac.in}



\begin{abstract}
One of the most celebrated problems in Euclidean Harmonic analysis is the Carleson's problem: determining the optimal regularity of the initial condition $f$ of the
Schr\"odinger equation given by
\begin{equation*}
\begin{cases}
	 i\frac{\partial u}{\partial t} =\Delta u\:,\:  (x,t) \in \mathbb{R}^n \times \mathbb{R} \\
	u(0,\cdot)=f\:, \text{ on } \mathbb{R}^n \:,
	\end{cases}
\end{equation*}
in terms of the index $\beta$ such that $f$ belongs to the inhomogeneous Sobolev space $H^\beta(\mathbb{R}^n)$, so that the solution of the Schr\"odinger operator $u$ converges pointwise to $f$, 
\begin{equation*}
\displaystyle\lim_{t \to 0+} u(x,t)=f(x), \text{ almost everywhere}.
\end{equation*}
Recently, the author considered the Carleson's problem for the Schr\"odinger equation with radial initial data on Damek-Ricci spaces and obtained the sharp bound up to the endpoint $\beta \ge 1/4$.

\medskip

Interpreting the above as convergence along vertical lines, in this article, we consider the problem of pointwise convergence via more general approach paths. By constructing a counter-example on the $3$-dimensional Real Hyperbolic space, we show that the solutions of the Schr\"odinger equation, unlike Harmonic functions or solutions of the Heat equation, do not admit any natural wide approach region. We then study their pointwise convergence properties on Damek-Ricci spaces along general curves that satisfy certain H\"older conditions and bilipschitz conditions in the distance from the identity and again obtain the sharp bound up to the endpoint $\beta \ge 1/4$. Certain Euclidean analogues are also obtained.
\end{abstract}

\maketitle
\tableofcontents

\section{Introduction}
One of the most celebrated problems in Euclidean Harmonic analysis is the Carleson's problem: determining the optimal regularity of the initial condition $f$ of the
Schr\"odinger equation given by
\begin{equation*}
\begin{cases}
	 i\frac{\partial u}{\partial t} =\Delta u\:,\:\:\:  (x,t) \in \mathbb{R}^n \times \mathbb{R}\:, \\
	u(0,\cdot)=f\:, \text{ on } \mathbb{R}^n \:,
	\end{cases}
\end{equation*}
in terms of the index $\beta$ such that $f$ belongs to the inhomogeneous Sobolev space $H^\beta(\mathbb{R}^n)$, so that the solution of the Schr\"odinger operator $u$ converges pointwise to $f$, 
\begin{equation*}
\displaystyle\lim_{t \to 0+} u(x,t)=f(x)\:,\:\text{ almost everywhere}.
\end{equation*}
Such a problem was first studied by Carleson \cite{C} and Dahlberg-Kenig \cite{DK} in dimension $1$, followed by several other experts in the field for arbitrary dimension (see \cite{Cowling, Sjolin, Vega, Bourgain, DGL, DZ} and references therein).

Recently, the author studied the Carleson's problem for solutions of the Schr\"odinger equation corresponding to the Laplace-Beltrami operator in the setting of Damek-Ricci spaces $S$ \cite{Dewan}. They are also known as Harmonic $NA$ groups. These spaces $S$ are non-unimodular, solvable extensions of Heisenberg type groups $N$, obtained by letting $A=\R^+$ act on $N$ by homogeneous dilations. The rank one Riemannian Symmetric spaces of noncompact type are the prototypical examples of (and in fact accounts for a very small subclass of the more general class of) Damek-Ricci spaces \cite[p. 643]{ADY}. Let $\Delta$ be the Laplace-Beltrami operator
on $S$ corresponding to the left-invariant Riemannian metric. Its $L^2$-spectrum is the half line $(-\infty,  -Q^2/4]$, where $Q$ is the homogeneous dimension of $N$. The Schr\"odinger equation on $S$ is given by
\begin{equation} \label{schrodinger}
\begin{cases}
	 i\frac{\partial u}{\partial t} =\Delta u\:,\:\:\:  (x,t) \in S \times \R \:,\\
	u(0,\cdot)=f\:,\: \text{ on } S \:.
	\end{cases}
\end{equation}
Then for a radial function $f$ belonging to the $L^2$-Schwartz class $\mathscr{S}^2(S)_o$ (for the definition, see (\ref{schwartz_defn})),
\begin{equation} \label{schrodinger_soln}
S_t f(x):= \int_{0}^\infty \varphi_\lambda(x)\:e^{it\left(\lambda^2 + \frac{Q^2}{4}\right)}\:\widehat{f}(\lambda)\: {|{\bf c}(\lambda)|}^{-2}\: d\lambda\:,
\end{equation}
is the solution to (\ref{schrodinger}), where $\varphi_\lambda$ are the spherical functions, $\widehat{f}$ is the Spherical Fourier transform of $f$ and ${\bf c}(\cdot)$ denotes the Harish-Chandra's ${\bf c}$-function. 

To quantify the Carleson's problem on $S$, for $\beta \ge 0$, we recall the fractional $L^2$-Sobolev spaces on $S$ for the special case of radial  functions \cite{APV}:
\begin{equation} \label{sobolev_space_defn}
H^\beta(S):=\left\{f \in L^2(S): {\|f\|}_{H^\beta(S)}:= {\left(\int_0^\infty {\left(\lambda^2 + \frac{Q^2}{4}\right)}^\beta {|\widehat{f}(\lambda)|}^2 {|{\bf c}(\lambda)|}^{-2} d\lambda\right)}^{1/2}< \infty\right\}.
\end{equation}

In \cite{Dewan}, we have proved that if $f$ is radial and belongs to the Sobolev space $H^\beta(S)$  for $\beta \ge 1/4$, then the solution to the Schr\"odinger equation with initial data $f$ converges to $f$, that is,
\begin{equation} \label{vertical_convergence}
\displaystyle\lim_{t \to 0+} S_tf(x)=f(x)\:,
\end{equation} 
for almost every $x \in S$, with respect to the left Haar measure. Moreover, the sharpness of the endpoint $\beta=1/4$ has been established in a subsequent paper \cite{DR}. Interpreting the convergence (\ref{vertical_convergence}) as convergence along vertical lines (in the product space $S \times (0,\infty)$), it is natural to ask what happens when the problem of (almost everywhere) pointwise convergence is considered along a wider approach region, for instance, the non-tangential convergence. 

For $\R^n$, this question regarding non-tangential convergence has been already addressed. On $\R^n$, for $\beta >n/2$, by Sobolev imbedding and standard arguments, the non-tangential convergence to the initial data holds. However, Sj\"ogren and Sj\"olin showed in \cite{SS} that the non-tangential convergence fails for $\beta \le n/2$. Now as their counter-example is a non-radial function, it is natural to ask when the initial data is more symmetric, say radial, then can we have pointwise convergence along wider approach regions. We answer this question in the negative by constructing a radial initial data on the $3$-dimensional Real Hyperbolic space $\mathbb{H}^3 \cong SL(2,\C)/SU(2)$, whose Schr\"odinger propagation blows up on geodesic annuli  along wide approach regions: 
\begin{theorem} \label{counter-example}
Assume that $\gamma :[0,\infty) \to [0,\infty)$ is a strictly increasing continuous function with $\gamma(0)=0$. Then given any compact geodesic annulus $\Sa \subset \mathbb{H}^3$, centered at the identity, there exists radial $f \in H^{1/2}(\mathbb{H}^3)$ such that its Schr\"odinger propagation $u$ is continuous in $\Sa \times (0,\infty)$ and
\begin{equation*}
\displaystyle\limsup_{\substack{(y,t) \to (x,0)\\ d(x,y)<\gamma(t),\:t>0}} |u(y,t)| = +\infty \:,\:\text{ for all } x \in \Sa \:.
\end{equation*}
\end{theorem}

Theorem \ref{counter-example} shows that the solutions of the Schr\"odinger equation, unlike Harmonic functions or solutions of the Heat equation, do not admit a natural wide approach region. Thus it becomes interesting to study the pointwise convergence of solutions along appropriate curves and its connection with the regularity of the initial data. Such problems have received considerable interest in recent years: in dimension one, by Cho-Lee-Vargas \cite{CLV}, Ding-Niu \cite{DN} and in higher dimensions, by Cao-Miao \cite{CM}, Minguill\'on \cite{Minguillon}.

In \cite{CM} and \cite{Minguillon}, the authors have shown that on $\R^n$, the condition $\beta> n/2(n+1)$ is sufficient to obtain almost everywhere pointwise convergence along certain families of curves. Firstly, their arguments crucially use dilation and the explicit knowledge of its interaction with the Fourier transform. In the case of a Riemannian manifold, where the volume measure neither satisfies any doubling property nor does it admit dilation (such as Damek-Ricci spaces), the above mentioned Euclidean machineries work no longer and thus it offers a fresh challenge. Secondly, it is interesting to ask when the initial condition is more symmetric, for instance, radial, whether one can lower the above mentioned regularity threshold and yet have pointwise convergence. 

In this light, we commence our exploration on Damek-Ricci spaces by identifying a natural family of curves to study. For $x \in S$, we consider curves $\gamma_x : \R \to S$ such that $\gamma_x(0)=x$. Now as we will be exclusively dealing with radial initial data, it is both natural and customary to identify all the points $x$ on a geodesic sphere centered at the identity $e$ of $S$, solely by their geodesic distance, say $s$ from $e$, that is, $s=d(e,x) \in [0,\infty)$. Here $d(\cdot,\cdot)$ is the inner metric on $S$ corresponding to the left-invariant Riemannian metric.  Then the solution of the Schr\"odinger equation with initial data $f \in \mathscr{S}^2(S)_o$, along a curve $\gamma_s$ is given by,
\begin{equation*} 
S_t f(\gamma_s(t))=S_t f(d(e,\gamma_s(t)))= \int_{0}^\infty \varphi_\lambda(d(e,\gamma_s(t)))\:e^{it\left(\lambda^2 + \frac{Q^2}{4}\right)}\:\widehat{f}(\lambda)\: {|{\bf c}(\lambda)|}^{-2}\: d\lambda\:.
\end{equation*}
Again keeping the notion of radiality in mind, we consider curves that satisfy the following natural analogues of the conditions considered in \cite{CLV}: there exist constants $C_j\: (j=1,2,3)$, independent of $s,s' \in [0,\infty)\:,\:t,t' \in \R$ and $\alpha \ge 0$, such that
\begin{eqnarray*}
&&(\mathscr{H}_1) \:\:\:\:\:\: \left|\:d(e,\gamma_s(t))-d(e,\gamma_s(t'))\:\right| \le C_1\: {|t-t'|}^\alpha\:, \\
&&(\mathscr{H}_2) \:\:\:\:\:\: C_2\:|s-s'| \le \left|\:d(e,\gamma_s(t))-d(e,\gamma_{s'}(t))\:\right| \le C_3\:|s-s'|\:.
\end{eqnarray*}
We note that by the triangle inequality, $(\mathscr{H}_1)$ is weaker than the condition:
\begin{equation*}
d\left(\gamma_s(t),\gamma_s(t')\right) \le C_1\: {|t-t'|}^\alpha\:,
\end{equation*}
and hence $\alpha$ is essentially the degree of tangential convergence. 

Clearly, the above curves generalize the classical case of vertical lines $\gamma_s(t)=s$, studied in \cite{Dewan} and thus one may expect to require more regularity on the initial data to guarantee pointwise convergence. Our result (Corollary \ref{cor1}) however, shows that the regularity threshold $\beta \ge 1/4$ (as in the case of vertical lines) is again sufficient to obtain almost everywhere pointwise convergence, even along such general curves. In fact, we prove more and this brings us to the definition of the homogeneous Sobolev spaces. For $\beta \ge 0$, the homogeneous Sobolev spaces for the special case of radial  functions (corresponding to the shifted Laplace-Beltrami operator $\tilde{\Delta}:=\Delta + \frac{Q^2}{4}$) are defined as,
\begin{equation*} 
\dot{H}^\beta(S):=\left\{f \in L^2(S): {\|f\|}_{\dot{H}^\beta(S)}:= {\left(\int_0^\infty \lambda^{2\beta}\: {|\widehat{f}(\lambda)|}^2 {|{\bf c}(\lambda)|}^{-2} d\lambda\right)}^{1/2}< \infty\right\}\:.
\end{equation*}
We now state our first result dealing with maximal estimates on annuli (in the following statement $A(\cdot)$ denotes the density function, see section $2$ for more details):
\begin{theorem} \label{thm1}
Let $0<r_1<r_2<\infty$. Assume that $\gamma$ satisfies $(\mathscr{H}_1)$ for $\alpha \in [\frac{1}{2}, 1]$ and $(\mathscr{H}_2)$ for $s,s' \in [0,r_2)$ and $t,t' \in [-T,T]$, for some fixed 
\begin{equation} \label{time_localization}
0<T < {\left(\frac{C_2r_1}{2C_1}\right)}^{1/\alpha}\:.
\end{equation}
Then we have for all $f \in \mathscr{S}^2(S)_o$\:,
\begin{equation} \label{thm1_estimate}
{\left(\bigintssss_{r_1}^{r_2} {\left(\displaystyle\sup_{t \in [-T,T]}\left|S_t f(d(e,\gamma_s(t)))\right|\right)}^2 A(s)\:ds\right)}^{1/2} \le C \: {\|f\|}_{\dot{H}^{1/4}(S)}\:,
\end{equation}
where the positive constant $C$ depends only on $r_1,r_2$ and the dimension of $S$.
\end{theorem}
The key idea in the proof of Theorem \ref{thm1} is to keep track of the oscillation afforded by both the spherical functions as well as the multiplier corresponding to the Schr\"odinger operator. But this becomes a challenging task as in the generality of Damek-Ricci spaces, no explicit expression of $\varphi_\lambda$ is known. In this regard, we crucially use a variation of the classical Harish-Chandra series expansion of $\varphi_\lambda$ which is only valid in compacts not containing the group identity. This forces us to consider specific time localizations (\ref{time_localization}) and moreover, to take the rather unorthodox approach of obtaining maximal estimates on annuli, in stead of balls. The oscillation of $\varphi_\lambda$ is realized by looking at the leading terms in the series expansion. The corresponding oscillatory parts of the linearized maximal function are taken care of by proceeding with an abstract $TT^*$ argument, followed by an oscillatory integral estimate and then finally applying the Pitt's inequality. The remaining part of the linearized maximal function is then dealt with by the pointwise decay of certain error terms appearing from the series expansion due to Anker et. al., which turns out to be of critical importance.

By standard arguments in the literature (for instance, see the proof of Theorem 5 of \cite{Sjolin}), Theorem \ref{thm1} yields pointwise convergence for solutions of the Schr\"odinger operator on annuli,
\begin{equation*}
A(r_1,r_2):=\left\{x \in S \mid r_1 < d(e,x) <r_2 \right\}\:,
\end{equation*} 
for all $0<r_1<r_2<\infty\:.$ Then as
$$A\left(\frac{1}{N},N\right) \:\big\uparrow \: S \setminus \{e\}\:,\:\text{ as } N \to \infty\:,$$
by a simple application of the Monotone Convergence Theorem, we get the desired result on pointwise convergence:
\begin{corollary}\label{cor1}
Let $\alpha \in [\frac{1}{2},1]$. Suppose that for every $s_0 \in [0,\infty)$, there exists a neighborhood $V$ $\subset [0,\infty) \times \R$ of $(s_0,0)$ such that $(\mathscr{H}_1)$ holds for $(s,t),(s,t') \in V$ and $(\mathscr{H}_2)$ holds for all $(s,t),(s',t) \in V$. Then for all $f \in \dot{H}^{1/4}(S)_o$ and hence for all $f \in H^{\beta}(S)_o$, with $\beta \ge 1/4$, we have
\begin{equation*}
\displaystyle\lim_{t \to 0} S_tf(\gamma_s(t))=f(s)\:,
\end{equation*}
for almost every $s \in [0,\infty)$, which also means
\begin{equation*}
\displaystyle\lim_{t \to 0} S_tf(\gamma_x(t))=f(x)\:,
\end{equation*}
for almost every $x \in S$, with respect to the left Haar measure.
\end{corollary}

\begin{remark} \label{remark1}
The sharpness of the regularity threshold $\beta=1/4$ in (\ref{thm1_estimate}), follows from the sharpness in the special case of vertical lines \cite[proof of Theorem 1.1, statement (i), pp. 21-23]{DR}.
\end{remark}

In this article, we include the case when $N$ is abelian and $\mathfrak{n}$ coincides with its center in the definition of $H$-type groups, as a degenerate case, so that the Real hyperbolic spaces are also included in our discussion (see \cite[pp. 209-210]{CDKR}).

We now return to $\R^n$, to address the issue raised earlier about lowering the required regularity threshold while specializing to radial initial data, for $n \ge 2$. For a radial function $f$ belonging to the Schwartz class $\mathscr{S}(\R^n)_o$, the solution of the Schr\"odinger equation on $\R^n$, 
\begin{equation} \label{schrodinger_euclidean}
\begin{cases}
	 i\frac{\partial u}{\partial t} =\Delta_{\R^n} u\:,\:  (x,t) \in \R^n \times \R \\
	u(0,\cdot)=f\:,\: \text{ on } \R^n \:,
	\end{cases}
\end{equation}
is given by
\begin{equation} \label{schrodinger_soln_euclidean}
\tilde{S}_t f(s):= \int_{0}^\infty \J_{\frac{n-2}{2}}(\lambda s)\:e^{it\lambda^2}\:\F f(\lambda)\: \lambda^{n-1}\: d\lambda\:,
\end{equation}
where $\J_{\frac{n-2}{2}}$ is the modified Bessel function of order $\frac{n-2}{2}$ (see section $2$ for more details), $\F f$ is the Euclidean Spherical Fourier transform of $f$ and $s=\|x\|$, is the distance of $x$ from the origin $o$. In this regard, we also formulate the conditions $(\mathscr{H}_1)$ and $(\mathscr{H}_2)$, in terms of the Euclidean norm $\|\cdot\|$ : there exist constants $C_j\: (j=4,5,6)$, independent of $s,s' \in [0,\infty)\:,\:t,t' \in \R$ and $\alpha \ge 0$, such that
\begin{eqnarray*}
&&(\mathscr{H}_3) \:\:\:\:\:\: \left|\:\|\gamma_s(t)\|-\|\gamma_s(t')\|\:\right| \le C_4\: {|t-t'|}^\alpha\:, \\
&&(\mathscr{H}_4) \:\:\:\:\:\: C_5\:|s-s'| \le \left|\:\|\gamma_s(t)\|-\|\gamma_{s'}(t)\|\:\right| \le C_6\:|s-s'|\:.
\end{eqnarray*}

It is interesting to observe that the arguments in the proof of Theorem \ref{thm1} can be carried out analogously for the special case of $\R^3$, but doing the same in the full generality of $\R^n$ seems to be a rather difficult task (see Remark \ref{special_remark} for more details). Nevertheless, taking a detour using local geometry of Riemannian manifolds and Fourier Analytic tools on Homogeneous spaces, we obtain the following analogue of Theorem \ref{thm1} for small annuli:
\begin{theorem} \label{thm2}
Let $0<\delta_1<\delta_2$ be sufficiently small. Assume that $\gamma$ satisfies $(\mathscr{H}_3)$ for $\alpha \in [\frac{1}{2},1]$ and $(\mathscr{H}_4)$ for $s,s' \in [0,\delta_2)$ and $t,t' \in [-T,T]$ for some fixed $$0<T <{\left(\frac{C_5\delta_1}{2C_4}\right)}^{1/\alpha}\:.$$ Then we have for all $f \in \mathscr{S}(\R^n)_o$,
\begin{equation} \label{thm2_estimate}
{\left(\bigintssss_{\delta_1}^{\delta_2} {\left(\displaystyle\sup_{t \in [-T,T]}\left|\tilde{S}_t f(\|\gamma_s(t)\|)\right|\right)}^2 \:s^{n-1}\:ds\right)}^{1/2} \le C \: {\|f\|}_{\dot{H}^{1/4}(\R^n)}\:,
\end{equation}
where the positive constant $C$ depends only on $\delta_1,\delta_2$ and $n$.
\end{theorem} 
For the proof of Theorem \ref{thm2}, we first consider the cases when the Fourier transform of the initial condition is supported in a neighborhood of the origin and when it is supported away from the origin. In the first case, the linearized maximal function is quite straightforward to control. The latter case however, is quite involved. In the large frequency regime, viewing $\R^n$ as the tangent space at the identity of a suitable Damek-Ricci space, we form a connection with Theorem \ref{thm1}:
\begin{itemize}
\item Using the fact that the Riemannian exponential map on a connected, simply-connected, complete, non-positively curved Riemannian manifold is a bijective local radial isometry, we can pushforward the curves in the small annuli (in $\R^n$) to the non-flat manifold, preserving the local geometric conditions.
\item The connection between the initial data and their Schr\"odinger propagations for $\R^n$ and $S$, is then obtained by means of the Bessel series expansion of $\varphi_\lambda$, repeated applications of the Abel transform, followed by utilizing a Schwartz multiplier corresponding to the ratio of the weights of the respective Plancherel measures.
\end{itemize}
Then Theorem \ref{thm2} follows by taking a resolution of identity. 

\begin{remark} \label{remark2}
The sharpness of the regularity threshold $\beta=1/4$ in (\ref{thm2_estimate}), follows from the sharpness in the special case of vertical lines \cite[pp. 55-58]{Sjolin2}.
\end{remark}
As an application of Theorem \ref{thm2}, we look at the family of curves whose norms satisfy,
\begin{equation} \label{curve_norm_condition}
\|\gamma_s(t)\|=s+C_7t^{1/2}\:,
\end{equation}
for some $C_7 \ge 0$, in some neighborhood of every $(s_0,0)$. Suitable dilations of annuli and the invariance of the norms of the curves (\ref{curve_norm_condition}) under a certain parabolic dilation in space and time, make it possible to apply Theorem \ref{thm2} to obtain the following result on pointwise convergence, almost everywhere on $\R^n$: 
\begin{theorem} \label{thm3}
Suppose that for every $s_0 \in [0,\infty)$, there exists a neighborhood $V \subset [0,\infty) \times \R$ of $(s_0,0)$ such that (\ref{curve_norm_condition}) holds for all $(s,t) \in V$. Then for all $f \in \dot{H}^{1/4}(\R^n)_o$ and hence for all $f \in H^{\beta}(\R^n)_o$, with $\beta \ge 1/4$, we have
\begin{equation*}
\displaystyle\lim_{t \to 0} S_tf(\gamma_s(t))=f(s)\:,
\end{equation*}
for almost every $s \in [0,\infty)$, which also means
\begin{equation*}
\displaystyle\lim_{t \to 0} S_tf(\gamma_x(t))=f(x)\:,
\end{equation*}
for almost every $x \in \R^n$.
\end{theorem}

\begin{remark}
Our Euclidean results (Theorems \ref{thm2} and \ref{thm3}) can be interpreted as variations of results in \cite{CM} and \cite{Minguillon}, as imposing more symmetry on the intial data allows us to work with much lower regularity. 
\end{remark}

This article is organized as follows. In section $2$, we recall certain aspects of Euclidean Fourier Analysis, the essential preliminaries about Damek-Ricci spaces and Spherical Fourier Analysis thereon and a useful oscillatory integral estimate. Theorem \ref{counter-example} is proved in section $3$. Theorem \ref{thm1} is proved in section $4$. Theorems \ref{thm2} and \ref{thm3} are proved in section $5$. Finally, we conclude by making some remarks and posing some new problems in section $6$.

Throughout, the symbols `c' and `C' will denote positive constants whose values may change on each occurrence. The enumerated constants $c_1,c_2, \dots$ or $C_1,C_2, \dots$ will however be fixed throughout. $\N$ will denote the set of positive integers. For a set $E$, $\chi_E$ will denote the indicator function on $E$,
\begin{equation*}
\chi_E(x):=\begin{cases}
	 1\:  \text{ if } x \in E  \\
	0\:  \text{ if } x \notin E \:.
	\end{cases}
\end{equation*}

 Two non-negative functions $f_1$ and $f_2$ will be said to satisfy, 
\begin{itemize}
\item $f_1 \lesssim f_2$ (resp. $f_1 \gtrsim f_2$) if there exists a constant $C \ge 1$, so that
\begin{equation*}
f_1 \le C f_2 \:,\text{ (resp. } f_1 \ge C f_2 \text{)}\:.
\end{equation*}
\item $f_1 \asymp f_2$ if there exist constants $C,C'>0$, so that
\begin{equation*}
C f_1 \le f_2 \le C' f_1\:.
\end{equation*}
\end{itemize}

\section{Preliminaries}
In this section, we recall some preliminaries and fix our notations.
\subsection{Fourier Analysis on $\R^n$:}
In this subsection, we recall some Euclidean Fourier Analysis, most of which can be found in \cite{Grafakos, SW}.
On $\R$, for ``nice" functions $f$, the Fourier transform $\tilde{f}$ is defined as
\begin{equation*}
\tilde{f}(\xi)= \int_\R f(x)\: e^{-ix\xi}\:dx\:.
\end{equation*}
An important inequality in one-dimensional Fourier Analysis is the Pitt's inequality:
\begin{lemma} \cite[p. 489]{Stein} \label{Pitt's_ineq}
One has the inequality
\begin{equation*}
{\left(\int_\R {\left|\tilde{f}(\xi)\right|}^2 \:{|\xi|}^{-2 \beta} d\xi\right)}^{1/2} \lesssim {\left(\int_\R {\left|f(x)\right|}^p\: {|x|}^{\beta_1 p} dx\right)}^{1/p}\:,
\end{equation*}
where $\beta_1 = \beta + \frac{1}{2}- \frac{1}{p}$ and the following two conditions are satisfied:
\begin{equation*}
0 \le \beta_1 < 1 - \frac{1}{p}\:,\: \text{ and } 0 \le \beta < \frac{1}{2}\:.
\end{equation*}
\end{lemma}
A $C^\infty$ function $f$ on $\R$ is called a Schwartz class function if 
\begin{equation*}
\left|{\left(\frac{d}{dx}\right)}^M f(x)\right| \lesssim {(1+|x|)}^{-N} \:, \text{ for any } M, N \in \N \cup \{0\}\:.
\end{equation*}
We denote by $\mathscr{S}(\R)$ the class of all such functions and $\mathscr{S}(\R)_{e}$ will denote the collection of all even Schwartz class functions on $\R$.

Let $\beta \in \C$ with $Re(\beta)>0$. The Riesz potential of order $\beta$ is the operator
\begin{equation*}
I_\beta= {(-\Delta_{\R})}^{-\beta/2}\:.
\end{equation*}
One can also write as,
\begin{equation*}
I_\beta(f)(x)=C_\beta \int_\R f(y)\: {|x-y|}^{\beta -1}\:dy\:,
\end{equation*}
for some $C_\beta>0$, whenever $f \in \mathscr{S}(\R)$. For $\beta >0$, the Fourier transform of the Riesz potential of $f \in \mathscr{S}(\R)$ satisfies the following identity:
\begin{equation} \label{riesz_identity}
(I_\beta(f))^\sim(\xi)= C_\beta {|\xi|}^{-\beta}\tilde{f}(\xi)\:.
\end{equation}

For a radial function $f$ in $\R^n$ (for $n \ge 2$), the Euclidean Spherical Fourier transform is defined as
\begin{equation*}
\mathscr{F}f(\lambda):= \int_0^\infty f(s) \J_{\frac{n-2}{2}}(\lambda s)\: s^{n-1}\: ds\:, 
\end{equation*}
where for all $\mu \ge 0$,
\begin{equation*}
\J_\mu(z)= 2^\mu \: \pi^{1/2} \: \Gamma\left(\mu + \frac{1}{2} \right) \frac{J_\mu(z)}{z^\mu},
\end{equation*}
and $J_\mu$ are the Bessel functions \cite[p. 154]{SW}.

The class of radial Schwartz class functions on $\R^n$, denoted by $\mathscr{S}(\R^n)_{o}$, is defined to be the collection of $f \in C^\infty(\R^n)_{o}$ such that 
\begin{equation*} 
\left|{\left(\frac{d}{ds}\right)}^M f(s)\right| \lesssim {(1+s)}^{-N} \:, \text{ for any } M, N \in \N \cup \{0\}\:,
\end{equation*}
where $s=\|x\|$, is the distance of $x$ from the origin. $\mathscr{F}: \mathscr{S}(\R^n)_{o} \to \mathscr{S}(\R)_{e}$
defines a topological isomorphism.

\subsection{Damek-Ricci spaces and spherical Fourier Analysis thereon:}
In this section, we will explain the notations and state relevant results on Damek-Ricci spaces. Most of these results can be found in \cite{ADY, APV, A}.

Let $\mathfrak n$ be a two-step real nilpotent Lie algebra equipped with an inner product $\langle, \rangle$. Let $\mathfrak{z}$ be the center of $\mathfrak n$ and $\mathfrak v$ its orthogonal complement. We say that $\mathfrak n$ is an $H$-type algebra if for every $Z\in \mathfrak z$ the map $J_Z: \mathfrak v \to \mathfrak v$ defined by
\begin{equation*}
\langle J_z X, Y \rangle = \langle [X, Y], Z \rangle, \:\:\:\: X, Y \in \mathfrak v
\end{equation*}
satisfies the condition $J_Z^2 = -|Z|^2I_{\mathfrak v}$, $I_{\mathfrak v}$ being the identity operator on $\mathfrak v$. A connected and simply connected Lie group $N$ is called an $H$-type group if its Lie algebra is $H$-type. Since $\mathfrak n$ is nilpotent, the exponential map is a diffeomorphism
and hence we can parametrize the elements in $N = \exp \mathfrak n$ by $(X, Z)$, for $X\in \mathfrak v, Z\in \mathfrak z$. It follows from the Baker-Campbell-Hausdorff formula that the group law in $N$ is given by
\begin{equation*}
\left(X, Z \right) \left(X', Z' \right) = \left(X+X', Z+Z'+ \frac{1}{2} [X, X']\right), \:\:\:\: X, X'\in \mathfrak v; ~ Z, Z'\in \mathfrak z.
\end{equation*}
The group $A = \R^+$ acts on an $H$-type group $N$ by nonisotropic dilation: $(X, Z) \mapsto (\sqrt{a}X, aZ)$. Let $S = NA$ be the semidirect product of $N$ and $A$ under the above action. Thus the multiplication in $S$ is given by
\begin{equation*}
\left(X, Z, a\right)\left(X', Z', a'\right) = \left(X+\sqrt aX', Z+aZ'+ \frac{\sqrt a}{2} [X, X'], aa' \right),
\end{equation*}
for $X, X'\in \mathfrak v; ~ Z, Z'\in \mathfrak z; a, a' \in \R^+$.
Then $S$ is a solvable, connected and simply connected Lie group having Lie algebra $\mathfrak s = \mathfrak v \oplus \mathfrak z \oplus \R$ with Lie bracket
\begin{equation*}
\left[\left(X, Z, l \right), \left(X', Z', l' \right)\right] = \left(\frac{1}{2}lX' - \frac{1}{2} l'X, lZ'-lZ + [X, X'], 0\right).
\end{equation*}

We write $na = (X, Z, a)$ for the element $\exp(X + Z)a, X\in \mathfrak v, Z \in \mathfrak z, a\in A$. We note
that for any $Z \in \mathfrak z$ with $|Z| = 1$, $J_Z^2 = -I_{\mathfrak v}$; that is, $J_Z$ defines a complex structure
on $\mathfrak v$ and hence $\mathfrak v$ is even dimensional. $m_\mv$ and $m_z$ will denote the dimension of $\mv$ and $\z$ respectively. Let $n$ and $Q$ denote dimension and the homogenous dimension of $S$ respectively:
\begin{equation*}
n=m_{\mv}+m_\z+1 \:\: \textit{ and } \:\: Q = \frac{m_\mv}{2} + m_\z.
\end{equation*}

The group $S$ is equipped with the left-invariant Riemannian metric induced by
\begin{equation*}
\langle (X,Z,l), (X',Z',l') \rangle = \langle X, X' \rangle + \langle Z, Z' \rangle + ll'
\end{equation*}
on $\mathfrak s$. For $x \in S$, we denote by $s=d(e,x)$, that is, the geodesic distance of $x$ from the identity $e$. Then the left Haar measure $dx$ of the group $S$ may be normalized so that
\begin{equation*}
dx= A(s)\:ds\:d\sigma(\omega)\:,
\end{equation*}
where $A$ is the density function given by,
\begin{equation} \label{density_function}
A(s)= 2^{m_\mv + m_\z} \:{\left(\sinh (s/2)\right)}^{m_\mv + m_\z}\: {\left(\cosh (s/2)\right)}^{m_\z} \:,
\end{equation}
and $d\sigma$ is the surface measure of the unit sphere.

A function $f: S \to \C$ is said to be radial if, for all $x$ in $S$, $f(x)$ depends only on the geodesic distance of $x$ from the identity $e$. If $f$ is radial, then
\begin{equation*}
\int_S f(x)~dx=\int_{0}^\infty f(s)~A(s)~ds\:.
\end{equation*}

We now recall the spherical functions on Damek-Ricci spaces. The spherical functions $\varphi_\lambda$ on $S$, for $\lambda \in \C$ are the radial eigenfunctions of the Laplace-Beltrami operator $\Delta$, satisfying the following normalization criterion
\begin{equation*}
\begin{cases}
 & \Delta \varphi_\lambda = - \left(\lambda^2 + \frac{Q^2}{4}\right) \varphi_\lambda  \\
& \varphi_\lambda(e)=1 \:.
\end{cases}
\end{equation*}
For all $\lambda \in \R$ and $x \in S$, the spherical functions satisfy
\begin{equation*}
\varphi_\lambda(x)=\varphi_\lambda(s)= \varphi_{-\lambda}(s)\:.
\end{equation*}
It also satisfies for all $\lambda \in \R$ and all $s \ge 0$:
\begin{equation} \label{phi_lambda_bound}
\left|\varphi_\lambda(s)\right| \le 1\:.
\end{equation}

The spherical functions are crucial as they help us define the Spherical Fourier transform of a ``nice" radial function $f$ (on $S$) in the following way:
\begin{equation*}
\widehat{f}(\lambda):= \int_S f(x) \varphi_\lambda(x) dx = \int_0^\infty f(s) \varphi_\lambda(s) A(s) ds\:.
\end{equation*}
The Harish-Chandra ${\bf c}$-function is defined as
\begin{equation*}
{\bf c}(\lambda)= \frac{2^{(Q-2i\lambda)} \Gamma(2i\lambda)}{\Gamma\left(\frac{Q+2i\lambda}{2}\right)} \frac{\Gamma\left(\frac{n}{2}\right)}{\Gamma\left(\frac{m_\mv + 4i\lambda+2}{4}\right)}\:,
\end{equation*}
for all $\lambda \in \R$. We will need the following pointwise estimates (see \cite[Lemma 4.8]{RS}):
\begin{equation} \label{plancherel_measure}
{|{\bf c}(\lambda)|}^{-2} \asymp \:{|\lambda|}^2 {\left(1+|\lambda|\right)}^{n-3}\:,
\end{equation}
and for $j \in \N \cup \{0\}$, the
derivative estimates (\cite[Lemma 4.2]{A}):
\begin{equation}\label{c-fn_derivative_estimates}
\left|\frac{d^j}{d \lambda^j}{|{\bf c}(\lambda)|}^{-2}\right| \lesssim_j {(1+|\lambda|)}^{n-1-j} \:, \:\: \lambda \in \R.
\end{equation}

One has the following inversion formula (when valid) for radial functions:
\begin{equation*}
f(x)= C \int_{0}^\infty \widehat{f}(\lambda)\varphi_\lambda(x) {|{\bf c}(\lambda)|}^{-2}
d\lambda\:,
\end{equation*}
where $C$ depends only on $m_\mv$ and $m_\z$. Moreover, the Spherical Fourier transform extends to an isometry from the space of radial $L^2$ functions on $S$ onto $L^2\left((0,\infty),C{|{\bf c}(\lambda)|}^{-2} d\lambda\right)$. 

The class of radial $L^2$-Schwartz class functions on $S$, denoted by $\mathscr{S}^2(S)_{o}$, is defined to be the collection of $f \in C^\infty(S)_{o}$ such that 
\begin{equation} \label{schwartz_defn}
\left|{\left(\frac{d}{ds}\right)}^M f(s)\right| \lesssim {(1+s)}^{-N} e^{-\frac{Q}{2}s}\:, \text{ for any } M, N \in \N \cup \{0\}\:,
\end{equation}
(see \cite[p. 652]{ADY}). One has the following commutative diagram, where every map is a topological isomorphism:
\begin{itemize}
\item $\mathscr{A}_{S,\R}$ is the Abel transform defined from $\mathscr{S}^2(S)_{o}$ to ${\mathscr{S}(\R)}_{e}$\:.
\item $\wedge$ denotes the Spherical Fourier transform from $\mathscr{S}^2(S)_{o}$ to ${\mathscr{S}(\R)}_{e}$\:.
\item $\sim$ denotes the 1-dimensional Euclidean Fourier transform from ${\mathscr{S}(\R)}_{e}$ to itself.
\end{itemize}

\[
\begin{tikzcd}[row sep=1.4cm,column sep=1.4cm]
\mathscr{S}^2(S)_{o}\arrow[r,"\mathscr{A}_{S,\R}"] \arrow[dr,swap,"\wedge"] & {\mathscr{S}(\R)}_{e}
\arrow[d,"\sim"] \\
&  {\mathscr{S}(\R)}_{e}&
\end{tikzcd}
\]

For more details regarding the Abel transform, we refer to \cite[p. 652-653]{ADY}. These have been generalized to the setting of Ch\'ebli-Trim\`eche Hypergroups \cite{BX}, whose simplest case is that of radial functions on $\R^n$. 

For the purpose of our article, we will require both the Bessel series expansion as well as another series expansion (similar to the Harish-Chandra series expansion) of the spherical functions. We first see an expansion of $\varphi_\lambda$ in terms of Bessel functions for points near the identity. But before that, let us define the following normalizing constant in terms of the Gamma functions,
\begin{equation*}
c_0 = 2^{m_\z}\: \pi^{-1/2}\: \frac{\Gamma(n/2)}{\Gamma((n-1)/2)}\:.
\end{equation*}
\begin{lemma}\cite[Theorem 3.1]{A} \label{bessel_series_expansion}
There exist $R_0, 2<R_0<2R_1$, such that for any $0 \le s \le R_0$, and any integer $M \ge 0$, and all $\lambda \in \R$, we have
\begin{equation*}
\varphi_\lambda(s)= c_0 {\left(\frac{s^{n-1}}{A(s)}\right)}^{1/2} \displaystyle\sum_{l=0}^M a_l(s)\J_{\frac{n-2}{2}+l}(\lambda s) s^{2l} + E_{M+1}(\lambda,s)\:,
\end{equation*}
where
\begin{equation*}
a_0 \equiv 1\:,\: |a_l(s)| \le C {(4R_1)}^{-l}\:,
\end{equation*}
and the error term has the following behaviour
\begin{equation*}
	\left|E_{M+1}(\lambda,s) \right| \le C_M \begin{cases}
	 s^{2(M+1)}  & \text{ if  }\: |\lambda s| \le 1 \\
	s^{2(M+1)} {|\lambda s|}^{-\left(\frac{n-1}{2} + M +1\right)} &\text{ if  }\: |\lambda s| > 1 \:.
	\end{cases}
\end{equation*}
Moreover, for every $0 \le s <2$, the series
\begin{equation*}
\varphi_\lambda(s)= c_0 {\left(\frac{s^{n-1}}{A(s)}\right)}^{1/2} \displaystyle\sum_{l=0}^\infty a_l(s)\J_{\frac{n-2}{2}+l}(\lambda s) s^{2l}\:,
\end{equation*}
is absolutely convergent.
\end{lemma}

For the asymptotic behaviour of the spherical functions when the distance from the identity is large, we look at the following series expansion \cite[pp. 735-736]{APV}:
\begin{equation} \label{anker_series_expansion}
\varphi_\lambda(s)= 2^{-m_\z/2} {A(s)}^{-1/2} \left\{{\bf c}(\lambda)  \displaystyle\sum_{\mu=0}^\infty \Gamma_\mu(\lambda) e^{(i \lambda-\mu) s} + {\bf c}(-\lambda) \displaystyle\sum_{\mu=0}^\infty \Gamma_\mu(-\lambda) e^{-(i\lambda + \mu) s}\right\}\:.
\end{equation}
The above series converges for $\lambda \in \R \setminus \{0\}$, uniformly on compacts not containing the group identity, where $\Gamma_0 \equiv 1$ and for $\mu \in \N$, one has the recursion formula,
\begin{equation*}
(\mu^2-2i\mu\lambda) \Gamma_\mu(\lambda) = \displaystyle\sum_{j=0}^{\mu -1}\omega_{\mu -j}\Gamma_{j}(\lambda)\:.
\end{equation*}
Then one has the following estimate on the coefficients \cite[Lemma 1]{APV}, for constants $C>0, d \ge 0$:
\begin{equation} \label{coefficient_estimate}
\left|\Gamma_\mu(\lambda)\right| \le C \mu^d {\left(1+|\lambda|\right)}^{-1}\:,
\end{equation}
for all $\lambda \in \R \setminus \{0\}, \mu \in \N$.

The relevant preliminaries for the degenerate case of the Real hyperbolic spaces can be found in \cite{Anker, ST, AP}. 

\subsection{Estimate of an Oscillatory integral}
We will also require the following oscillatory integral estimate, which can be obtained by proceeding exactly as in the proof of Lemma $2.1$ in \cite{DN}: 
\begin{lemma} \label{oscillatory_integral_lemma}
Let $0<\delta_1 <\delta_2< \infty$. Let $\kappa(\cdot,\cdot) : [0,\infty) \times \R \to [0,\infty)$ satisfy
\begin{eqnarray*}
&&|\kappa(s,t)-\kappa(s,t')| \lesssim \: {|t-t'|}^\alpha \:,\\
&&|\kappa(s,t)-\kappa(s',t)| \: \asymp \: {|s-s'|} \:,
\end{eqnarray*}
for $\alpha \in [\frac{1}{2},1]$, $s,s' \in [0,\delta_2)$ and $t,t' \in [-T,T]$ with some fixed $T>0$. Assume that $t(\cdot): (\delta_1,\delta_2) \to [-T,T]$ is a measurable function. If $\mu \in C^\infty_c[0,\infty)$, then for any $C>0$,
\begin{equation*}
\left|\int_0^\infty e^{i\left\{\lambda\left(\kappa(s',t(s'))-\kappa(s,t(s))\right)+(t(s')-t(s))\left(\lambda^2+C\right)\right\}}\: \lambda^{-\frac{1}{2}}\: \mu\left(\frac{\lambda}{N}\right)\:d\lambda \right| \lesssim \frac{1}{{|s-s'|}^{1/2}}\:,
\end{equation*}
for all $s,s' \in (\delta_1,\delta_2)$ and $N \in \N$. The implicit constant in the conclusion depends only on $\mu$. 
\end{lemma}

\section{A counter-example for wide approach regions}
In this section we make use of several standard notions in Harmonic Analysis on Riemannian Symmetric spaces. For unexplained terminologies and more details, we refer to \cite{Helgason}.

Let $G=SL(2,\C)$ and $K$ be its maximal compact subgroup $SU(2)$. Here,
\begin{equation*}
A = \left\{a_t=
\begin{pmatrix}
e^t &0 \\
0 &e^{-t}
\end{pmatrix} : t \in \R
\right\}\:,
\end{equation*}
and $\Sigma_+$ consists of a single root $\alpha$ that occurs with multiplicity $2$.  We normalize $\alpha$ so that $\alpha(log\: a_t)=t$. Every $\lambda \in \C$ can be identified with an element in $\mathfrak{a}^*_{\C}$ by $\lambda = \lambda \alpha$. We see that in this identification (the half-sum of positive roots counted with multiplicity) $\rho=1,\: Q=2\rho=2$ and
\begin{equation*}
G/K \cong \mathbb{H}^3(-1)\:,
\end{equation*}
where $\mathbb{H}^3(-1)$ is the $3$-dimensional Real Hyperbolic space with constant sectional curvature $-1$. Then by \cite[p. 432, Theorem 5.7]{Helgason} we get the following information for $G/K$:
\begin{itemize}
\item For $\lambda \in (0,\infty)$, the Spherical function $\varphi_\lambda$ is given by,
\begin{equation*}
\varphi_\lambda(s) = \frac{\sin(\lambda s)}{\lambda \sinh(s)}\:.
\end{equation*}
\item Plancherel measure: $\lambda^2 \: d\lambda$, where $d\lambda$ is the Lebesgue measure on $(0,\infty)$.
\end{itemize}
We now present the proof of Theorem \ref{counter-example}.
\begin{proof}[Proof of Theorem \ref{counter-example}]
The compact geodesic annulus $\Sa$ is of the form,
\begin{equation*}
\Sa = \left\{x \in \mathbb{H}^3 \:|\: c_1 \le d(e,x)\le c_2\right\}\:,
\end{equation*}
for some positive constants $c_1<c_2$. Then depending on $\gamma, c_1$ and $c_2$, there exist positive constants $c_3,c_4,c_5$ such that $c_3 \in (0,1)$ and 
\begin{equation*}
0<c_4<c_1<c_2<c_5\:,
\end{equation*}
so that
\begin{equation*}
\displaystyle\bigcup_{x \in \Sa} \left\{y \in \mathbb{H}^3 \:|\: d(x,y) < \gamma(t)\:,\:0<t<c_3\right\} \subset  \left\{z \in \mathbb{H}^3 \:|\: c_4 \le d(e,z) \le c_5\right\}=:\Sa'\:. 
\end{equation*}
Let $\{x_j\}_{j=1}^\infty$ be a countable dense subset of $\Sa'$ and $\{t_j\}_{j=1}^\infty$ be any sequence of positive real numbers such that 
\begin{equation*}
c_3 > t_1 > t_2 > \dots > 0 \:,\: \text{ and } t_j \downarrow 0\:.
\end{equation*}
We claim that given any $x \in \Sa$, there exists a subsequence $\{(x_{j_l},t_{j_l})\}_{l=1}^\infty$ such that
\begin{equation} \label{counter-eq1}
(x_{j_l},t_{j_l}) \to (x,0)\text{ as } l \to \infty\:,\text{ while } d(x,x_{j_l}) < \gamma(t_{j_l})\:.
\end{equation}
Indeed, fixing $x \in \Sa$, we set $t_{j_1}:=t_1$ and then note that there exists $x_{j_1}(\ne x) \in \mathscr{B}(x,\gamma(t_{j_1}))$, the geodesic ball with center $x$ and radius $\gamma(t_{j_1})$. Then as $\gamma$ is continuous and strictly increasing, there exists $t_{j_1}' \in (0,t_{j_1})$ such that $\gamma(t_{j_1}')=d(x,x_{j_1})$. We now set $t_{j_2}:=\max\{t_j < t_{j_1}'\}$ and again note that there exists $x_{j_2}(\ne x) \in \mathscr{B}(x,\gamma(t_{j_2}))$. Then again as $\gamma$ is continuous and strictly increasing, there exists $t_{j_2}' \in (0,t_{j_2})$ such that $\gamma(t_{j_2}')=d(x,x_{j_2})$. Proceeding like this, we get a subsequence $\{(x_{j_l},t_{j_l})\}_{l=1}^\infty$ such that 
\begin{equation*}
d(x,x_{j_l}) < \gamma(t_{j_l}) \:,
\end{equation*}
and combining it with the fact that $t_{j_l} \downarrow 0$, it follows that 
\begin{equation*}
(x_{j_l},t_{j_l}) \to (x,0)\text{ as } l \to \infty\:.
\end{equation*}
This completes the proof of the claim (\ref{counter-eq1}).

We next let $\{r_j\}_{j=1}^\infty$ and $\{R_j\}_{j=1}^\infty$ be two sequences such that 
\begin{equation*}
3=r_1<R_1<r_2<R_2<\dots \:,
\end{equation*} 
and set
\begin{equation*}
\Omega_j:= \left\{\lambda \in (0,\infty)\:\mid\:r_j < \lambda < R_j\right\}\:.
\end{equation*}
Corresponding to the enumeration $\{x_j\}_{j=1}^\infty$, we consider radial $f \in L^2(\mathbb{H}^3)$, with the spherical Fourier transform
\begin{equation*}
\widehat{f}(\lambda)= \displaystyle\sum_{j=1}^\infty \chi_{\Omega_j}(\lambda)\:\lambda^{-1}\:{\left(\log \lambda\right)}^{-3/4}\:\varphi_\lambda(s_j)\:e^{-it_j(\lambda^2+\rho^2)}\:,
\end{equation*}
where $s_j=d(e,x_j)$. Then using the expression of $\varphi_\lambda$ and the fact that $s_j \ge c_4$, for all $j$, we get that
\begin{eqnarray*}
{\|f\|}^2_{H^{1/2}(\mathbb{H}^3)}= \int_0^\infty {\left(\lambda^2 + \rho^2\right)}^{1/2} {|\widehat{f}(\lambda)|}^2 \:\lambda^2\: d\lambda \lesssim \int_3^\infty \frac{d\lambda}{\lambda (\log \lambda)^{3/2}} < +\infty\:.
\end{eqnarray*}
Therefore $f \in H^{1/2}(\mathbb H^3)$. We now make our choice of  $\{r_j\}_{j=1}^\infty$ and $\{R_j\}_{j=1}^\infty$ more precise. We first choose $r_1=3,\:R_1=5$. Given $r_1,R_1,\dots, r_{j-1},R_{j-1}$, we then choose $r_j > R_{j-1}$ such that for $k<j$, one has for some $c_6>0$,
\begin{equation} \label{counter-eq2}
r^2_j \ge \frac{c_6 2^j}{|t_k-t_j|}\:,
\end{equation}
and
\begin{equation} \label{counter-eq3}
|t_k-t_j|r_j > s_j+s_k+1\:.
\end{equation}
We also set $R_j=r^M_j$, where $M>0$ is chosen large. For $m \in \N$, now let
\begin{equation*}
u_m(s,t):= \int_3^{R_m} \varphi_\lambda(s)\:e^{it(\lambda^2+\rho^2)}\:\widehat{f}(\lambda)\:\lambda^2\:d\lambda\:.
\end{equation*}
Now using the explicit expression of $\varphi_\lambda$, we have for all $(x,t) \in \Sa' \times (0,\infty)$,
\begin{eqnarray} \label{counter-eq4}
u_m(x,t)=u_m(s,t)&=& \displaystyle\sum_{j=1}^m \int_{\Omega_j} \lambda^{-1}\:{\left(\log \lambda\right)}^{-3/4}\:\varphi_\lambda(s)\:\varphi_\lambda(s_j)\:e^{i(t-t_j)(\lambda^2+\rho^2)}\:\lambda^2\:d\lambda \nonumber \\
&=& -\frac{1}{4} \displaystyle\sum_{j=1}^m \frac{e^{i(t-t_j)\rho^2}}{\sinh(s)\sinh(s_j)}\int_{\Omega_j} \lambda^{-1}\:{\left(\log \lambda\right)}^{-3/4}\:e^{i\lambda(s+s_j)}\:e^{i(t-t_j)\lambda^2}\:d\lambda \nonumber\\
&& +\frac{1}{4} \displaystyle\sum_{j=1}^m \frac{e^{i(t-t_j)\rho^2}}{\sinh(s)\sinh(s_j)}\int_{\Omega_j} \lambda^{-1}\:{\left(\log \lambda\right)}^{-3/4}\:e^{i\lambda(s-s_j)}\:e^{i(t-t_j)\lambda^2}\:d\lambda \nonumber\\
&& +\frac{1}{4} \displaystyle\sum_{j=1}^m \frac{e^{i(t-t_j)\rho^2}}{\sinh(s)\sinh(s_j)}\int_{\Omega_j} \lambda^{-1}\:{\left(\log \lambda\right)}^{-3/4}\:e^{-i\lambda(s-s_j)}\:e^{i(t-t_j)\lambda^2}\:d\lambda \nonumber\\
&& -\frac{1}{4} \displaystyle\sum_{j=1}^m \frac{e^{i(t-t_j)\rho^2}}{\sinh(s)\sinh(s_j)}\int_{\Omega_j} \lambda^{-1}\:{\left(\log \lambda\right)}^{-3/4}\:e^{-i\lambda(s+s_j)}\:e^{i(t-t_j)\lambda^2}\:d\lambda \nonumber\\
&=& u_{m,1}(s,t) + u_{m,2}(s,t) + u_{m,3}(s,t) + u_{m,4}(s,t)\:.
\end{eqnarray}
For each $k \in \N$ and $m>k$, our goal is to estimate $\left|u_m(x_k,t_k)\right|$. We first focus on
\begin{equation*}
u_{m,1}(s_k,t_k) = -\frac{1}{4} \displaystyle\sum_{j=1}^m \frac{e^{i(t_k-t_j)\rho^2}}{\sinh(s_k)\sinh(s_j)} A_j(s_k,t_k)\:.
\end{equation*}
As $s_j,s_k \ge c_4$, for all $j,k$, we first have
\begin{eqnarray*}
\left|\displaystyle\sum_{j=1}^{k-1} \frac{e^{i(t_k-t_j)\rho^2}}{\sinh(s_k)\sinh(s_j)} A_j(s_k,t_k)\right| & \lesssim & \displaystyle\sum_{j=1}^{k-1} \int_{\Omega_j} \lambda^{-1}\:{\left(\log \lambda\right)}^{-3/4}\:d\lambda \\
&\le & \int_3^{R_{k-1}} \lambda^{-1}\:{\left(\log \lambda\right)}^{-3/4}\:d\lambda \\
&\lesssim & {\left(\log r_k\right)}^{1/4} \:.
\end{eqnarray*}
Next, an integration by parts and an application of the Fundamental theorem of calculus yield
\begin{eqnarray*}
\left|A_k(s_k,t_k)\right| &=& \left|\int_{\Omega_k}\lambda^{-1}\:{\left(\log \lambda\right)}^{-3/4}\:e^{2i\lambda s_k}\:d\lambda\right| \\
& \lesssim & r^{-1}_k\:{\left(\log r_k\right)}^{-3/4} + \int_{\Omega_k} \left|\frac{d}{d\lambda}\left(\lambda^{-1}\:{\left(\log \lambda\right)}^{-3/4}\right)\right|\:d\lambda \\
&=& 2 r^{-1}_k\:{\left(\log r_k\right)}^{-3/4} - R^{-1}_k\:{\left(\log R_k\right)}^{-3/4} \\
& \lesssim & r^{-1}_k\:{\left(\log r_k\right)}^{-3/4}\:.
\end{eqnarray*} 
For $j>k$, we note that
\begin{equation*}
A_j(s_k,t_k)= \int_{\Omega_j} \lambda^{-1}\:{\left(\log \lambda\right)}^{-3/4}\:e^{iF_1(\lambda)}\:d\lambda\:,
\end{equation*}
where, 
\begin{equation*}
F_1(\lambda)= (s_k+s_j)\lambda + (t_k-t_j)\lambda^2\:.
\end{equation*}
Thus
\begin{eqnarray} \label{counter-eq5}
&&F'_1(\lambda)= s_k+s_j + 2(t_k-t_j)\lambda\:,\nonumber\\
&&F''_1(\lambda)= 2(t_k-t_j)\:.
\end{eqnarray}
Next using (\ref{counter-eq3}), for $\lambda \in \Omega_j$, it follows that
\begin{equation} \label{counter-eq6}
\left|F'_1(\lambda)\right| \ge  2|t_k-t_j|\lambda - (s_k+s_j) > |t_k-t_j|\lambda > |t_k-t_j|r_j\:.  
\end{equation}
Now an integration by parts yields
\begin{eqnarray*}
&&\int_{\Omega_j} \lambda^{-1}\:{\left(\log \lambda\right)}^{-3/4}\:e^{iF_1(\lambda)}\:d\lambda \\ &=& \int_{r_j}^{R_j} \frac{1}{\lambda\:{\left(\log \lambda\right)}^{3/4} iF'_1(\lambda)}iF'_1(\lambda)\:e^{iF_1(\lambda)}\:d\lambda \\
&=& \left[\frac{1}{\lambda\:{\left(\log \lambda\right)}^{3/4} iF'_1(\lambda)}\:e^{iF_1(\lambda)}\right]^{R_j}_{r_j} - \int_{r_j}^{R_j} \frac{d}{d\lambda}\left(\frac{1}{\lambda\:{\left(\log \lambda\right)}^{3/4} iF'_1(\lambda)}\right) \:e^{iF_1(\lambda)}\:d\lambda \\
&=& I-J\:.
\end{eqnarray*}
Now invoking (\ref{counter-eq2}) and (\ref{counter-eq6}), it follows that
\begin{equation*}
|I| \le \frac{2}{|t_k-t_j|r^2_j} \le \left(\frac{2}{c_6}\right)2^{-j}\:.
\end{equation*}
Using (\ref{counter-eq5}) and (\ref{counter-eq6}), we also have
\begin{eqnarray*}
\left|\frac{d}{d\lambda}\left(\frac{1}{\lambda\:{\left(\log \lambda\right)}^{3/4} iF'_1(\lambda)}\right)\right| \le \frac{2}{\lambda^2 \left|F'_1(\lambda)\right|} + \frac{\left|F''_1(\lambda)\right|}{\lambda {\left|F'_1(\lambda)\right|}^2} \le \frac{4}{|t_k-t_j|\lambda^3}\:.
\end{eqnarray*}
Hence,
\begin{equation*}
|J| \le \frac{2}{|t_k-t_j|r^2_j} \le \left(\frac{2}{c_6}\right)2^{-j}\:.
\end{equation*}
Thus setting, $c_7=4/c_6$, we note that for $j>k$,
\begin{equation*} 
\left|A_j(s_k,t_k)\right| \le c_7 2^{-j}\:.
\end{equation*}
Therefore for $m>k$,
\begin{equation*} 
\left|u_{m,1}(s_k,t_k)\right| \lesssim {\left(\log r_k\right)}^{1/4} + r^{-1}_k\:{\left(\log r_k\right)}^{-3/4} + \displaystyle\sum_{j=k+1}^m 2^{-j} \lesssim {\left(\log r_k\right)}^{1/4}\:.
\end{equation*}
Similarly, we get that
\begin{equation*} 
\left|u_{m,4}(s_k,t_k)\right| \lesssim {\left(\log r_k\right)}^{1/4}\:.
\end{equation*}
We now consider
\begin{equation*}
u_{m,2}(s_k,t_k)=\frac{1}{4} \displaystyle\sum_{j=1}^m \frac{e^{i(t_k-t_j)\rho^2}}{\sinh(s_k)\sinh(s_j)}B_j(s_k,t_k)\:.
\end{equation*}
Then again proceeding as in the case of $u_{m,1}$, we get that
\begin{eqnarray*}
\left|\displaystyle\sum_{j=1}^{k-1} \frac{e^{i(t_k-t_j)\rho^2}}{\sinh(s_k)\sinh(s_j)} B_j(s_k,t_k)\right| 
&\lesssim & {\left(\log r_k\right)}^{1/4} \:.
\end{eqnarray*}
Next, we see that for $M>0$ sufficiently large,
\begin{eqnarray*}
B_k(s_k,t_k)&=& \int_{r_k}^{R_k} \lambda^{-1}\:{\left(\log \lambda\right)}^{-3/4}\:d\lambda \\
&=& 4\left({\left(\log R_k\right)}^{1/4}- {\left(\log r_k\right)}^{1/4}\right) \\
&=& 4\left(1-\frac{1}{M^{1/4}}\right) {\left(\log R_k\right)}^{1/4} \\
&\gtrsim& {\left(\log R_k\right)}^{1/4}\:.
\end{eqnarray*}
For $j>k$, we note that
\begin{equation*}
B_j(s_k,t_k)= \int_{\Omega_j} \lambda^{-1}\:{\left(\log \lambda\right)}^{-3/4}\:e^{iF_2(\lambda)}\:d\lambda\:,
\end{equation*}
where, 
\begin{equation*}
F_2(\lambda)= (s_k-s_j)\lambda + (t_k-t_j)\lambda^2\:.
\end{equation*}
Thus
\begin{eqnarray*} 
&&F'_2(\lambda)= (s_k-s_j) + 2(t_k-t_j)\lambda\:,\nonumber\\
&&F''_2(\lambda)= 2(t_k-t_j)\:.
\end{eqnarray*}
So for $j>k$, either $s_k=s_j$ or $s_k \ne s_j$. But in either case, 
\begin{equation*} 
\left|F'_2(\lambda)\right| > |t_k-t_j|\lambda > |t_k-t_j|r_j\:.  
\end{equation*}
The above is trivial if $s_k=s_j$. If $s_k \ne s_j$ proceeding as before using (\ref{counter-eq3}), we get the above inequality. Then again proceeding as before using the above inequality and (\ref{counter-eq2}), we get that
\begin{equation*} 
\left|B_j(s_k,t_k)\right| \le c_7 2^{-j}\:.
\end{equation*}
Similarly, writing 
\begin{equation*}
u_{m,3}(s_k,t_k)=\frac{1}{4} \displaystyle\sum_{j=1}^m \frac{e^{i(t_k-t_j)\rho^2}}{\sinh(s_k)\sinh(s_j)}D_j(s_k,t_k)\:,
\end{equation*}
we would get that
\begin{eqnarray*}
&&\left|\displaystyle\sum_{j=1}^{k-1} \frac{e^{i(t_k-t_j)\rho^2}}{\sinh(s_k)\sinh(s_j)} D_j(s_k,t_k)\right| 
\lesssim  {\left(\log r_k\right)}^{1/4} \:,\\
&& D_k(s_k,t_k) \gtrsim {\left(\log R_k\right)}^{1/4} \:,\\
&& \left|D_j(s_k,t_k)\right| \le c_7 2^{-j}\:,\text{ for } j>k\:.
\end{eqnarray*}
Thus for $m>k$ and $M>0$ sufficiently large, using the decomposition (\ref{counter-eq4}), the fact that all $s_j,s_k \in [c_4,c_5]$ and the above estimates, we get that
\begin{eqnarray} \label{counter-eq7}
\left|u_m(x_k,t_k)\right| &\ge & \left|u_{m,2}(s_k,t_k)+ u_{m,3}(s_k,t_k)\right| - \left|u_{m,1}(s_k,t_k)\right| - \left|u_{m,4}(s_k,t_k)\right| \nonumber\\
& \ge & \frac{1}{4 \sinh^2(s_k)}\left(B_k(s_k,t_k)+D_k(s_k,t_k)\right) - \frac{1}{4}\left|\displaystyle\sum_{j=1}^{k-1} \frac{e^{i(t_k-t_j)\rho^2}}{\sinh(s_k)\sinh(s_j)} B_j(s_k,t_k)\right| \nonumber\\
&& - \frac{1}{4}\left|\displaystyle\sum_{j=1}^{k-1} \frac{e^{i(t_k-t_j)\rho^2}}{\sinh(s_k)\sinh(s_j)} D_j(s_k,t_k)\right|  - \frac{1}{4}\left|\displaystyle\sum_{j=k+1}^{m} \frac{e^{i(t_k-t_j)\rho^2}}{\sinh(s_k)\sinh(s_j)} B_j(s_k,t_k)\right| \nonumber\\
&& - \frac{1}{4}\left|\displaystyle\sum_{j=k+1}^{m} \frac{e^{i(t_k-t_j)\rho^2}}{\sinh(s_k)\sinh(s_j)} D_j(s_k,t_k)\right| - \left|u_{m,1}(s_k,t_k)\right| - \left|u_{m,4}(s_k,t_k)\right| \nonumber\\
&\gtrsim& {\left(\log R_k\right)}^{1/4} - {\left(\log r_k\right)}^{1/4} -1 \nonumber\\
&\gtrsim& {\left(\log R_k\right)}^{1/4}\:.
\end{eqnarray}
We now show that for all $m \in \N$, $u_m$ is continuous on $\Sa' \times (0,\infty)$. For $(s,t) \in [c_4,c_5] \times (0,\infty)$, let $(\varepsilon,h) \in \R^2$ be such that $s + \varepsilon \ge c_4/2$. We write,
\begin{equation*}
\left|u_m(s+\varepsilon,t+h)-u_m(s,t)\right| \le \left|u_m(s+\varepsilon,t+h)-u_m(s,t+h)\right| + \left|u_m(s,t+h)-u_m(s,t)\right| \:.
\end{equation*}
Then as $[3,R_m]$ is compact, by the Dominated Convergence Theorem, we get that
\begin{equation*}
\left|u_m(s+\varepsilon,t+h)-u_m(s,t+h)\right| \le \int_3^{R_m} \left|\varphi_\lambda(s+\varepsilon)- \varphi_\lambda(s)\right| \left|\widehat{f}(\lambda)\right|\:\lambda^2\:d\lambda \to 0 \text{ as } |\varepsilon| \to 0\:,
\end{equation*}
and also
\begin{equation*}
\left|u_m(s,t+h)-u_m(s,t)\right| \le \int_3^{R_m} \left|e^{ih(\lambda^2+\rho^2)}- 1\right| \left|\widehat{f}(\lambda)\right|\:\lambda^2\:d\lambda \to 0 \text{ as } |h| \to 0\:.
\end{equation*}
Now if we can show that $u_m \to u$ locally uniformly on $\Sa' \times (0,\infty)$, then it follows that $u$ is also continuous on $\Sa' \times (0,\infty)$. To see the local uniform convergence, let $L$ be a compact subset of $\Sa' \times (0,\infty)$. We now claim that there exists $j_0 \in \N$ such that for all $j>j_0$ and all $(x,t) \in L$, \begin{equation} \label{counter-eq8}
r^2_j \ge \frac{c_6 2^j}{|t-t_j|}\:,
\end{equation}
and
\begin{equation} \label{counter-eq9}
|t-t_j|r_j > s_j+s+1\:,
\end{equation}
which basically means that $(s_k,t_k)$ in (\ref{counter-eq2}) and (\ref{counter-eq3}) can be replaced by $(s,t)$ corresponding to $(x,t) \in L$, when $j$ is large enough. To see the above claims, we note that by compactness of $L$, there exists $j' \in \N$ such that for all $j \ge j'$,
\begin{equation*}
t_j < \inf \:\{t| (x,t) \in L\}\:.
\end{equation*}
Then as the sequence $\{t_j\}_{j=1}^\infty$ is monotonically decreasing, combining it with (\ref{counter-eq2}), we get that for all $j>j'$ and all $(x,t) \in L$,
\begin{equation*}
|t-t_j| > |t_{j'}-t_j| \ge \frac{c_6 2^j}{r^2_j}\:.
\end{equation*}
and also 
\begin{equation*}
\frac{|t-t_j|}{s_j+s+1} > \frac{t-t_{j'}}{2c_5+1}\:.
\end{equation*}
Now as $r_j \uparrow \infty$, there exists $j'' \in \N$ such that for all $j > j''$,
\begin{equation*}
\frac{t-t_{j'}}{2c_5+1} > \frac{1}{r_j}\:.
\end{equation*}
Thus setting $j_0:=\max\{j',j''\}$, the claims (\ref{counter-eq8}) and (\ref{counter-eq9}) are established. Then using (\ref{counter-eq8}) and (\ref{counter-eq9}), proceeding as before, we get that there exists $c_8>0$, such that for all $j>j_0$ and all $(x,t) \in L$,
\begin{equation*}
\left| \int_{\Omega_j} \varphi_\lambda(s)\:e^{it(\lambda^2+\rho^2)}\:\widehat{f}(\lambda)\:\lambda^2\:d\lambda\right| \le c_8 2^{-j}\:.
\end{equation*}
Thus for all $m>j_0$ and all $(x,t) \in L$,
\begin{eqnarray*}
\left|u_m(x,t)-u(x,t)\right| &=& \left| \int_{R_m}^\infty \varphi_\lambda(s)\:e^{it(\lambda^2+\rho^2)}\:\widehat{f}(\lambda)\:\lambda^2\:d\lambda\right| \\
&\le & \displaystyle\sum_{j=m+1}^\infty
\left| \int_{\Omega_j} \varphi_\lambda(s)\:e^{it(\lambda^2+\rho^2)}\:\widehat{f}(\lambda)\:\lambda^2\:d\lambda\right| \\
& \le & c_8 \displaystyle\sum_{j=m+1}^\infty 2^{-j} = c_8 2^{-m}\:.
\end{eqnarray*}
This gives the desired local uniform convergence and consequently, the continuity of $u$ on $\Sa' \times (0,\infty)$. Finally for any $(x_k,t_k)$ and sufficiently large $m_k$, we have by the above inequality and (\ref{counter-eq8}),
\begin{equation*}
\left|u(x_k,t_k)\right| \ge \left|u_{m_k}(x_k,t_k)\right| -c_8 \:\gtrsim\: {\left(\log R_k\right)}^{1/4}\:. 
\end{equation*} 
Now choose and fix $x \in \Sa$. We then consider the subsequence $\{(x_{j_l},t_{j_l})\}_{l=1}^\infty$ satisfying the property (\ref{counter-eq1}) and obtain
\begin{equation*}
\displaystyle\limsup_{\substack{(y,t) \to (x,0)\\ d(x,y)<\gamma(t),\:t>0}} |u(y,t)| \ge \displaystyle\lim_{l \to \infty} \left|u(x_{j_l},t_{j_l})\right| \gtrsim\: \displaystyle\lim_{l \to \infty} {\left(\log R_{j_l}\right)}^{1/4}\:=\: +\infty\:.
\end{equation*}
As $x \in \Sa$ was arbitrary, this completes the proof of Theorem \ref{counter-example}.
\end{proof}

\section{Results on Damek-Ricci spaces}
In this section, we aim to prove Theorem \ref{thm1}. But first we see the following lemma.
\begin{lemma} \label{curve_growth_lemma}
Under the hypothesis of Theorem \ref{thm1}, we have for all $s \in (r_1,r_2)$ and $t \in [-T,T]$,
\begin{equation*}
\frac{C_2}{2}s \:< \: d(e,\gamma_s(t)) \: < \frac{3C_3}{2}s\:.
\end{equation*}
\end{lemma}
\begin{proof}
We first note by $(\mathscr{H}_2)$ that
\begin{equation} \label{curve_growth_lemma_eq1}
C_2s\: \le \: \left|d(e,\gamma_s(t)) - d(e,\gamma_0(t))\right|\: \le \: C_3s\:.
\end{equation}
Next by $(\mathscr{H}_1)$, we have
\begin{equation} \label{curve_growth_lemma_eq2}
d(e,\gamma_0(t)) \:=\:\left|d(e,\gamma_0(t)) - d(e,\gamma_0(0))\right|\: \le \: C_1{|t|}^\alpha \:\le\: C_1T^\alpha\:.
\end{equation}
Therefore, combining (\ref{curve_growth_lemma_eq1}) and (\ref{curve_growth_lemma_eq2}), we get that
\begin{equation*}
C_2s-C_1T^\alpha \:\le\: d(e,\gamma_s(t)) \:\le\: C_3s+C_1T^\alpha\:.
\end{equation*}
Now as $0<T < {\left(\frac{C_2r_1}{2C_1}\right)}^{1/\alpha}$, we see that
\begin{equation*}
C_2s-C_1T^\alpha >  C_2 \left(s-\frac{r_1}{2}\right) > \frac{C_2}{2}s\:. 
\end{equation*}
Similarly, 
\begin{equation*}
C_3s+C_1T^\alpha < C_3s+ C_2 \left(\frac{r_1}{2}\right) \le C_3 \left(s+\frac{r_1}{2}\right) < \frac{3C_3}{2}s\:.
\end{equation*}
This completes the proof.
\end{proof}

\begin{proof}[Proof of Theorem \ref{thm1}]
In order to prove the theorem, it suffices to prove the following estimate in terms of the linearized maximal function,
\begin{equation} \label{thm1_pf_eq1}
{\left(\bigintssss_{r_1}^{r_2} {\left|T_\gamma f(s)\right|}^2 A(s)\:ds\right)}^{1/2} \lesssim  {\|f\|}_{\dot{H}^{1/4}(S)}\:,
\end{equation}
where,
\begin{equation} \label{linearized_max_fn}
T_\gamma f(s) := \int_{0}^\infty \varphi_\lambda(d(e,\gamma_s(t(s))))\:e^{it(s)\left(\lambda^2 + \frac{Q^2}{4}\right)}\:\widehat{f}(\lambda)\: {|{\bf c}(\lambda)|}^{-2}\: d\lambda \:,
\end{equation} 
and $t(\cdot): (r_1,r_2) \to [-T,T]$ is a measurable function. Now by Lemma \ref{curve_growth_lemma}, we have for all $s \in (r_1,r_2)$, 
\begin{equation*}
\frac{C_2}{2}r_1 < \frac{C_2}{2}s < d(e,\gamma_s(t(s))) < \frac{3C_3}{2}s < \frac{3C_3}{2}r_2   \:.
\end{equation*}
Then by invoking the  series expansion (\ref{anker_series_expansion}) of $\varphi_\lambda$, for $s \in (r_1, r_2)$ and $\lambda > 0$, we get that
\begin{eqnarray} \label{thm1_pf_eq2}
\varphi_\lambda (d(e,\gamma_s(t(s)))) &=& 2^{-m_\z/2} {A(d(e,\gamma_s(t(s))))}^{-1/2} \left\{{\bf c}(\lambda)  e^{i \lambda d(e,\gamma_s(t(s)))} + {\bf c}(-\lambda) e^{-i\lambda  d(e,\gamma_s(t(s)))}\right\} \nonumber\\
&& + \mathscr{E}(\lambda,d(e,\gamma_s(t(s))))\:,
\end{eqnarray}
where
\begin{eqnarray*}
\mathscr{E}(\lambda,d(e,\gamma_s(t(s))))&=& 2^{-m_\z/2} {A(d(e,\gamma_s(t(s))))}^{-1/2} \\
&& \times  \left\{{\bf c}(\lambda)  \displaystyle\sum_{\mu=1}^\infty \Gamma_\mu(\lambda) e^{(i \lambda-\mu) d(e,\gamma_s(t(s)))} + {\bf c}(-\lambda) \displaystyle\sum_{\mu=1}^\infty \Gamma_\mu(-\lambda) e^{-(i\lambda + \mu) d(e,\gamma_s(t(s)))}\right\} \:,
\end{eqnarray*}
and thus again using Lemma \ref{curve_growth_lemma}, the local growth asymptotics of the density function (\ref{density_function}) and the estimate (\ref{coefficient_estimate}) on the coefficients $\Gamma_\mu$, it follows that
\begin{equation} \label{thm1_pf_eq3}
\left|\mathscr{E}(\lambda,d(e,\gamma_s(t(s))))\right| \lesssim {A(s)}^{-1/2} \left|{\bf c}(\lambda)\right| {(1+\lambda)}^{-1} \:.
\end{equation}
Hence for $ s \in (r_1,r_2)$ and $\lambda >0$, in accordance to (\ref{thm1_pf_eq2}), we decompose $T_\gamma$ as,
\begin{eqnarray*}
&&T_\gamma f(s) \\
&=& 2^{-m_\z/2} {A(d(e,\gamma_s(t(s))))}^{-1/2} \bigintssss_0^\infty  {\bf c}(\lambda) \: e^{i\left\{\lambda d(e,\gamma_s(t(s))) + t(s)\left(\lambda^2 + \frac{Q^2}{4}\right)\right\}} \:\widehat{f}(\lambda)\: {|{\bf c}(\lambda)|}^{-2}\: d\lambda \\
&&+ 2^{-m_\z/2} {A(d(e,\gamma_s(t(s))))}^{-1/2} \bigintssss_0^\infty  {\bf c}(-\lambda) \: e^{i\left\{-\lambda d(e,\gamma_s(t(s))) + t(s)\left(\lambda^2 + \frac{Q^2}{4}\right)\right\}} \:\widehat{f}(\lambda)\: {|{\bf c}(\lambda)|}^{-2}\: d\lambda \\
&&+ \bigintssss_0^\infty \mathscr{E}(\lambda,d(e,\gamma_s(t(s)))) \:e^{it(s)\left(\lambda^2 + \frac{Q^2}{4}\right)}\:\widehat{f}(\lambda)\: {|{\bf c}(\lambda)|}^{-2}\: d\lambda \\
&=& T_{\gamma,1}f(s)\:+\: T_{\gamma,2}f(s)\: + \:T_{\gamma,3}f(s)\:.
\end{eqnarray*}
The arguments for $T_{\gamma,1}$ and $T_{\gamma,2}$ are similar and hence we work out the details only for $T_{\gamma,1}$. We aim to show that
\begin{equation} \label{thm1_pf_eq4}
{\left(\bigintssss_{r_1}^{r_2} {\left|T_{\gamma,1} f(s)\right|}^2 A(s)\:ds\right)}^{1/2} \lesssim  {\left(\int_0^\infty \lambda^{1/2}\: {|\widehat{f}(\lambda)|}^2 {|{\bf c}(\lambda)|}^{-2} d\lambda\right)}^{1/2}\:.
\end{equation}
Now
\begin{eqnarray*}
&& T_{\gamma,1} f(s)\: A(s)^{1/2} \\
&=& 2^{-m_\z/2} {A(d(e,\gamma_s(t(s))))}^{-1/2} A(s)^{1/2} \bigintssss_0^\infty  {\bf c}(\lambda)  e^{i\left\{\lambda d(e,\gamma_s(t(s))) + t(s)\left(\lambda^2 + \frac{Q^2}{4}\right)\right\}} \widehat{f}(\lambda) {|{\bf c}(\lambda)|}^{-2} d\lambda \\
&=& 2^{-m_\z/2} {A(d(e,\gamma_s(t(s))))}^{-1/2} A(s)^{1/2} \bigintssss_0^\infty  {\bf c}(\lambda)  e^{i\left\{\lambda d(e,\gamma_s(t(s))) + t(s)\left(\lambda^2 + \frac{Q^2}{4}\right)\right\}} g(\lambda)\lambda^{-1/4}{|{\bf c}(\lambda)|}^{-1} d\lambda,
\end{eqnarray*}
where
\begin{equation*}
g(\lambda)= \widehat{f}(\lambda)\:\lambda^{1/4}\:{|{\bf c}(\lambda)|}^{-1}\:.
\end{equation*}
Then setting,
\begin{eqnarray*}
P_\gamma g(s):={A(d(e,\gamma_s(t(s))))}^{-1/2} {A(s)}^{1/2} \bigintssss_{0}^\infty  {\bf c}(\lambda) e^{i\left\{\lambda d(e,\gamma_s(t(s)))) + t(s)\left(\lambda^2 + \frac{Q^2}{4}\right)\right\}}g(\lambda) \lambda^{-1/4} {|{\bf c}(\lambda)|}^{-1} d\lambda \:,
\end{eqnarray*}
that is,
\begin{equation*}
2^{m_\z/2}\: T_{\gamma,1}f(s)\: A(s)^{1/2} = P_\gamma g(s)\:,
\end{equation*}
we note that proving (\ref{thm1_pf_eq4}) is equivalent to proving
\begin{equation} \label{thm1_pf_eq5}
{\left(\int_{r_1}^{r_2} {\left|P_\gamma g(s)\right|}^2\:ds\right)}^{1/2} \lesssim {\left(\int_0^\infty  {|g(\lambda)|}^2 \:d\lambda\right)}^{1/2}\:.
\end{equation}
Taking $\rho \in C^\infty_c[0,\infty)$ real-valued such that $\rho(\lambda)=1$ if $\lambda \le 1$ and $\rho(\lambda)=0$ if $\lambda \ge 2$, for $N>2$ and $s \in (r_1,r_2)$, we set
\begin{eqnarray*}
P_{\gamma,N} g(s)&:=& {A(d(e,\gamma_s(t(s))))}^{-1/2} {A(s)}^{1/2} \\
&&\times \bigintssss_{0}^\infty  {\bf c}(\lambda) e^{i\left\{\lambda d(e,\gamma_s(t(s)))) + t(s)\left(\lambda^2 + \frac{Q^2}{4}\right)\right\}}\rho\left(\frac{\lambda}{N}\right)g(\lambda) \lambda^{-1/4} {|{\bf c}(\lambda)|}^{-1} d\lambda \:.
\end{eqnarray*}
Now for $\eta \in C^\infty_c(r_1,r_2)$ and $\lambda >0$, setting
\begin{eqnarray*}
P^*_{\gamma,N} \eta(\lambda)&:=& \overline{{\bf c}(\lambda)} \:\rho\left(\frac{\lambda}{N}\right) \lambda^{-1/4}\:{|{\bf c}(\lambda)|}^{-1} \\
&&\times \bigintssss_{r_1}^{r_2}  {A(d(e,\gamma_s(t(s))))}^{-1/2} {A(s)}^{1/2}  e^{-i\left\{\lambda d(e,\gamma_s(t(s)))) + t(s)\left(\lambda^2 + \frac{Q^2}{4}\right)\right\}} \eta(s) ds\:,
\end{eqnarray*}
it is easy to see that
\begin{equation*}
\int_{r_1}^{r_2} P_{\gamma,N}\psi(s)\: \overline{\eta(s)}\: ds = \int_0^\infty \psi(\lambda)\: \overline{P^*_{\gamma,N} \eta(\lambda)}\: d\lambda\:,
\end{equation*}
holds for all $\eta \in C^\infty_c(r_1,r_2)$ and $\psi \in L^2(0,\infty)$ having suitable decay at infinity. Thus it suffices to prove that 
\begin{equation} \label{thm1_pf_eq6}
{\left(\int_0^\infty {\left|P^*_{\gamma,N} h(\lambda)\right|}^2 d\lambda\right)}^{1/2} \lesssim {\left(\int_{r_1}^{r_2} {|h(s)|}^2\: ds\right)}^{1/2}\:,
\end{equation}
for all $h \in C^\infty_c(r_1,r_2)$, with the implicit constant independent of $N$, as then letting $N \to \infty$, we can obtain the estimate (\ref{thm1_pf_eq5}). Now by Fubini's theorem,
\begin{eqnarray*}
&& \int_0^\infty {\left|P^*_{\gamma,N} h(\lambda)\right|}^2 d\lambda \\
&=& \int_{r_1}^{r_2}\int_{r_1}^{r_2} I_N(s,s')
{A(d(e,\gamma_s(t(s))))}^{-\frac{1}{2}} {A(s)}^{\frac{1}{2}} {A(d(e,\gamma_{s'}(t(s'))))}^{-\frac{1}{2}} {A(s')}^{\frac{1}{2}} h(s) \overline{h(s')} ds\: ds',
\end{eqnarray*}
where 
\begin{equation*}
I_N(s,s') =  \int_{0}^\infty e^{i\left\{\lambda\left(d(e,\gamma_{s'}(t(s')))-d(e,\gamma_{s}(t(s)))\right)+(t(s')-t(s))\left(\lambda^2 + \frac{Q^2}{4}\right)\right\}} \lambda^{-1/2} \rho\left(\frac{\lambda}{N}\right)^2\: d\lambda \:.
\end{equation*}
Then by Lemmata \ref{oscillatory_integral_lemma}, \ref{curve_growth_lemma} and the local growth asymptotics of the density function, we get that
\begin{equation*}
\int_0^\infty {\left|P^*_{\gamma,N} h(\lambda)\right|}^2 d\lambda \lesssim  \int_{r_1}^{r_2}\int_{r_1}^{r_2} \frac{1}{{|s-s'|}^{1/2}}\: |h(s)|\: |h(s')|\: ds\: ds',
\end{equation*}
with the implicit constant independent of $N$. As $h \in C^\infty_c(r_1,r_2)$, we can think of it as an even $C^\infty_c$ function supported in $(-r_2,-r_1) \sqcup (r_1,r_2)$. We can therefore write the last integral as an one dimensional Riesz potential and apply (\ref{riesz_identity}) to get that for some positive number $c$, 
\begin{eqnarray*}
\int_{r_1}^{r_2}\int_{r_1}^{r_2} \frac{1}{{|s-s'|}^{1/2}}\: |h(s)|\: |h(s')|\: ds\: ds' &=& \int_{0}^{\infty}\int_0^\infty \frac{1}{{|s-s'|}^{1/2}}\: |h(s)|\: |h(s')|\: ds\: ds' \\
&=& c\int_{0}^\infty I_{1/2} (|h|)(s)\:|h(s)|\:ds \\
&=& c \int_{\R} {|\xi|}^{-\frac{1}{2}}\: {\left|\tilde{h}(\xi)\right|}^2\: d\xi \:.
\end{eqnarray*}
Then by Pitt's inequality (as in our case $\beta=\frac{1}{4}, p=2$ and hence $\beta_1= \beta + \frac{1}{2} - \frac{1}{p} =\frac{1}{4}$, both the conditions of Lemma \ref{Pitt's_ineq} are satisfied),
\begin{eqnarray*}
 \int_{\R} {|\xi|}^{-\frac{1}{2}}\: {\left|\tilde{h}(\xi)\right|}^2\: d\xi
& \lesssim & \int_{\R} {|h(x)|}^2\: {|x|}^{\frac{1}{2}}\: dx \\
&=& c \int_{r_1}^{r_2} {|h(s)|}^{2}\: s^{\frac{1}{2}}\: ds \\
&\le & c \:r^{1/2}_2 \: {\|h\|}^2_{{L^2}(r_1,r_2)}\:.
\end{eqnarray*}
Thus we get (\ref{thm1_pf_eq4}). Similarly, 
\begin{equation} \label{thm1_pf_eq7}
{\left(\bigintssss_{r_1}^{r_2} {\left|T_{\gamma,2} f(s)\right|}^2 A(s)\:ds\right)}^{1/2} \lesssim  {\left(\int_0^\infty \lambda^{1/2}\: {|\widehat{f}(\lambda)|}^2 {|{\bf c}(\lambda)|}^{-2} d\lambda\right)}^{1/2}\:.
\end{equation}
Finally for $T_{\gamma,3}$, using the estimate on the error term (\ref{thm1_pf_eq3}) and Cauchy-Schwarz inequality, we obtain
\begin{eqnarray*}
\left|T_{\gamma,3} f(s)\right| & \lesssim & {A(s)}^{-1/2}\: \int_0^\infty  {(1+\lambda)}^{-1}\: \left|\widehat{f}(\lambda)\right|\: {|{\bf c}(\lambda)|}^{-1}\: d\lambda \\
& \le & {A(s)}^{-1/2}\: \left[\int_0^1 \left|\widehat{f}(\lambda)\right|\: {|{\bf c}(\lambda)|}^{-1}\: d\lambda + \int_1^\infty \lambda^{-1}\:\left|\widehat{f}(\lambda)\right|\: {|{\bf c}(\lambda)|}^{-1}\: d\lambda  \right] \\
& \le & {A(s)}^{-1/2}\: {\|f\|}_{\dot{H}^{1/4}(S)} \: \left[{\left(\int_0^1 \frac{d\lambda}{\lambda^{1/2}}\right)}^{1/2} + {\left(\int_1^\infty \frac{d\lambda}{\lambda^{5/2}}\right)}^{1/2}\right] \\
& \lesssim & {A(s)}^{-1/2}\: {\|f\|}_{\dot{H}^{1/4}(S)} \:. 
\end{eqnarray*}
Thus
\begin{equation} \label{thm1_pf_eq8}
{\left(\bigintssss_{r_1}^{r_2} {\left|T_{\gamma,3} f(s)\right|}^2 A(s)\:ds\right)}^{1/2} \lesssim {(r_2-r_1)}^{1/2}\:{\|f\|}_{\dot{H}^{1/4}(S)}  \:.
\end{equation}
Then combining (\ref{thm1_pf_eq4}), (\ref{thm1_pf_eq7}) and (\ref{thm1_pf_eq8}), we complete the proof.
\end{proof}

\begin{remark} \label{special_remark}
In the special case of $\mathbb{H}^3(-1)$, the $SL(2,\C)$-invariant Riemannian measure on $\mathbb{H}^3(-1)$ is $\sinh^2 s \: ds \: dk$, where $ds$ and $dk$ are the Lebesgue measure on $(0,\infty)$ and the Haar measure on $SU(2)$ respectively. As noted in section $3$, in this case, the spherical function is given by the explicit expression:
\begin{equation*}
\varphi_\lambda(s)=\frac{1}{2i \sinh(s)}\left(\frac{e^{i\lambda s}}{\lambda}+\frac{e^{-i\lambda s}}{\lambda}\right)\:,
\end{equation*}
and hence the proof of Theorem \ref{thm1} can be carried out with relative ease (as the corresponding error term is zero). The situation is exactly similar in the classical $3$-dimensional Euclidean space, as the spherical function on $\R^3$ is given by,
\begin{equation*}
\frac{\sin(\lambda s)}{\lambda s}= \frac{1}{2i s}\left(\frac{e^{i\lambda s}}{\lambda}+\frac{e^{-i\lambda s}}{\lambda}\right)\:.
\end{equation*}   
In the full generality of $\R^n$ however, the spherical function is somewhat lesser explicit and is given by $\J_{\frac{n-2}{2}}$, the modified Bessel function of order $\frac{n-2}{2}$. Although the oscillatory nature of $\J_{\frac{n-2}{2}}$ is well-studied, yet proceeding analogously as in the arguments of the proof of Theorem \ref{thm1} does not seem to be an easy task.
\end{remark}

Aiming to eventually obtain some Euclidean results, we now focus on `small annuli'. Let 
\begin{equation*}
0<\delta_1<\delta_2 \le (2R_0)/(3C_3)\:,
\end{equation*}
where $R_0$ is as in Lemma \ref{bessel_series_expansion}. As in Theorem \ref{thm1}, we also assume that $\gamma$ satisfies $(\mathscr{H}_1)$ for $\alpha \in [\frac{1}{2}, 1]$ and $(\mathscr{H}_2)$ for $s,s' \in [0,\delta_2)$ and $t,t' \in [-T,T]$, for some fixed $T>0$, such that 
\begin{equation*}
T^\alpha < (C_2\delta_1)/(2C_1)\:.
\end{equation*}
Then by Lemma \ref{curve_growth_lemma}, we have for all $s \in (\delta_1,\delta_2)$ and $t \in [-T,T]$,
\begin{equation*}
d(e,\gamma_s(t)) < \frac{3C_3}{2}s < \frac{3C_3}{2}\delta_2 \le R_0\:. 
\end{equation*}
Thus for $f \in \mathscr{S}^2(S)_o$, we can also decompose the linearized maximal function (\ref{linearized_max_fn}) in terms of the Bessel series expansion of $\varphi_\lambda$ (by taking $M=0$ in Lemma \ref{bessel_series_expansion}),
\begin{eqnarray*} 
T_\gamma f(s) &=& \int_{0}^\infty \varphi_\lambda(d(e,\gamma_s(t(s))))\:e^{it(s)\left(\lambda^2 + \frac{Q^2}{4}\right)}\:\widehat{f}(\lambda)\: {|{\bf c}(\lambda)|}^{-2}\: d\lambda \nonumber\\
&=& c_0 {\left(\frac{d(e,\gamma_s(t(s)))^{n-1}}{A(d(e,\gamma_s(t(s))))}\right)}^{1/2} \int_0^\infty \J_{\frac{n-2}{2}}(\lambda d(e,\gamma_s(t(s))))\:e^{it(s)\left(\lambda^2 + \frac{Q^2}{4}\right)}\:\widehat{f}(\lambda)\: {|{\bf c}(\lambda)|}^{-2}\: d\lambda \nonumber\\
&& + \int_{0}^\infty E(\lambda, d(e,\gamma_s(t(s))))\:e^{it(s)\left(\lambda^2 + \frac{Q^2}{4}\right)}\:\widehat{f}(\lambda)\: {|{\bf c}(\lambda)|}^{-2}\: d\lambda \nonumber\\
&=& T_{\gamma,4} f(s) + T_{\gamma,5} f(s)\:.
\end{eqnarray*} 

We obtain the following boundedness result for $T_{\gamma,4}$, which will be crucial for the proof of Theorem \ref{thm2} in the next section:

\begin{corollary} \label{small_annuli_cor}
Let $0<\delta_1<\delta_2 \le (2R_0)/(3C_3)$, where $R_0$ is as in Lemma \ref{bessel_series_expansion}. Assume that $\gamma$ satisfies $(\mathscr{H}_1)$ for $\alpha \in [\frac{1}{2}, 1]$ and $(\mathscr{H}_2)$ for $s,s' \in [0,\delta_2)$ and $t,t' \in [-T,T]$, for some fixed $T>0$, such that $T^\alpha < (C_2\delta_1)/(2C_1)$. Then we have for all $f \in \mathscr{S}^2(S)_o$\:,
\begin{equation*} 
{\left(\bigintssss_{\delta_1}^{\delta_2} {\left| T_{\gamma,4} f(s)\right|}^2 A(s)\:ds\right)}^{1/2} \lesssim \: {\|f\|}_{\dot{H}^{1/4}(S)}\:,
\end{equation*}
where the implicit constant depends only on $\delta_1,\delta_2$ and the dimension of $S$.
\end{corollary}

\begin{proof}
By Theorem \ref{thm1}, it suffices to prove that
\begin{equation} \label{cor4.1_eq1}
{\left(\bigintssss_{\delta_1}^{\delta_2} {\left| T_{\gamma,5} f(s)\right|}^2 A(s)\:ds\right)}^{1/2} \lesssim \: {\|f\|}_{\dot{H}^{1/4}(S)}\:,
\end{equation}
where the implicit constant depends only on $\delta_1,\delta_2$ and the dimension of $S$. For $\lambda \ge 0$, we recall the estimates on the error term in Lemma \ref{bessel_series_expansion} (for $M=0$),
\begin{equation*}
	\left|E(\lambda,d(e,\gamma_s(t(s)))) \right| \lesssim \begin{cases}
	 d(e,\gamma_s(t(s)))^2  & \text{ if  }\: \lambda d(e,\gamma_s(t(s))) \le 1 \\
	d(e,\gamma_s(t(s)))^2\: \{\lambda d(e,\gamma_s(t(s)))\}^{-\left(\frac{n+1}{2}\right)} &\text{ if  }\: \lambda d(e,\gamma_s(t(s))) > 1 \:.
	\end{cases}
\end{equation*}
Then using the above pointwise bounds, Lemma \ref{curve_growth_lemma} and Cauchy-Schwarz inequality, we obtain
\begin{eqnarray*}
\left| T_{\gamma,5} f(s)\right| & \lesssim & d(e,\gamma_s(t(s)))^2 \bigintssss_0^{1/d(e,\gamma_s(t(s)))} \left|\widehat{f}(\lambda)\right|\: {|{\bf c}(\lambda)|}^{-2}\: d\lambda \\
&& + \frac{1}{(d(e,\gamma_s(t(s))))^{\left(\frac{n-3}{2}\right)}} \bigintssss_{1/d(e,\gamma_s(t(s)))}^\infty \lambda^{-\left(\frac{n+1}{2}\right)} \left|\widehat{f}(\lambda)\right|\: {|{\bf c}(\lambda)|}^{-2}\: d\lambda \\
&\lesssim & {\|f\|}_{\dot{H}^{1/4}(S)} \left[s^2 I_1(s)^{1/2} + \frac{1}{s^{\left(\frac{n-3}{2}\right)}} I_2(s)^{1/2} \right]\:,
\end{eqnarray*} 
where 
\begin{equation*}
I_1(s) = \bigintssss_0^{1/d(e,\gamma_s(t(s)))} \frac{{|{\bf c}(\lambda)|}^{-2}\: d\lambda}{\lambda^{1/2}} \:,\:\: \text{ and } \:\: I_2(s) = \bigintssss_{1/d(e,\gamma_s(t(s)))}^\infty \frac{{|{\bf c}(\lambda)|}^{-2}\: d\lambda}{\lambda^{n+\frac{3}{2}}} \:. 
\end{equation*}
Now first using the asymptotics of ${|{\bf c}(\lambda)|}^{-2}$ given in (\ref{plancherel_measure}) for small frequency and Lemma \ref{curve_growth_lemma}, we get that 
\begin{equation*}
I_1(s) \lesssim \bigintssss_0^{1/d(e,\gamma_s(t(s)))} \lambda^{3/2}\: d\lambda \: \lesssim \: \frac{1}{s^{5/2}}\:.
\end{equation*}
Similarly the asymptotics of ${|{\bf c}(\lambda)|}^{-2}$ given in (\ref{plancherel_measure}) for large frequency and Lemma \ref{curve_growth_lemma} yield
\begin{equation*}
I_2(s) \lesssim \bigintssss_{1/d(e,\gamma_s(t(s)))}^\infty \:\:\frac{d\lambda}{\lambda^{5/2}} \: \lesssim \: s^{3/2} \:.
\end{equation*}
Hence, for $s \in (\delta_1, \delta_2)$,
\begin{equation*}
\left| T_{\gamma,5} f(s)\right| \lesssim {\|f\|}_{\dot{H}^{1/4}(S)} \: s^{\left(\frac{9-2n}{4}\right)}\:, 
\end{equation*}
and thus using the local growth asymptotics of the density function, we obtain
\begin{equation*}
\bigintssss_{\delta_1}^{\delta_2} {\left| T_{\gamma,5} f(s)\right|}^2 A(s)\:ds \lesssim  {\|f\|}^2_{\dot{H}^{1/4}(S)}\: \bigintssss_{\delta_1}^{\delta_2} s^{7/2}\:ds \:\lesssim_{\delta_1,\delta_2} \: {\|f\|}^2_{\dot{H}^{1/4}(S)}\:.
\end{equation*}
This completes the proof of Corollary \ref{small_annuli_cor}.
\end{proof}

\section{Results on $\R^n$}
Our aim in this section is to prove Theorems \ref{thm2} and \ref{thm3}.

For $f \in \mathscr{S}(\R^n)_o$, with $n \ge 2$, the linearization of the maximal function appearing in Theorem \ref{thm2} is given by, 
\begin{equation} \label{linearized_max_fn_Rn}
\tilde{T}_\gamma f(s) := \int_{0}^\infty \J_{\frac{n-2}{2}}(\lambda \|\gamma_s(t(s))\|)\:e^{it(s)\lambda^2}\:\mathscr{F}f(\lambda)\: \lambda^{n-1}\: d\lambda \:,
\end{equation} 
where $t(\cdot): (\delta_1,\delta_2) \to [-T,T]$ is a measurable function. We first consider the case when $\mathscr{F}f$ is supported near the origin.
\begin{lemma} \label{euclidean_small_freq_lemma}
Under the hypothesis of Theorem \ref{thm2}, we have for all $f \in \mathscr{S}(\R^n)_o$ with $Supp\left(\mathscr{F}f\right) \subset [0,\Lambda]$ for some $\Lambda>0$,
\begin{equation*} 
{\left(\bigintssss_{\delta_1}^{\delta_2} {\left|\tilde{T}_\gamma f(s)\right|}^2 \:s^{n-1}\:ds\right)}^{1/2} \lesssim \: {\|f\|}_{\dot{H}^{1/4}(\R^n)}\:,
\end{equation*}
where the implicit constant depends only on $\Lambda,\delta_1,\delta_2$ and $n$.
\end{lemma}   
\begin{proof}
The Lemma follows from the boundedness of the modified Bessel functions and Cauchy-Schwarz inequality. Indeed, for $s \in (\delta_1,\delta_2)$,
\begin{equation*}
\left|\tilde{T}_\gamma f(s)\right| \lesssim \int_0^\Lambda \left|\mathscr{F}f(\lambda)\right|\: \lambda^{n-1}\: d\lambda \le {\|f\|}_{\dot{H}^{1/4}(\R^n)}{\left(\int_0^\Lambda \lambda^{n-\frac{3}{2}}\: d\lambda\right)}^{1/2} \lesssim_{\Lambda} {\|f\|}_{\dot{H}^{1/4}(\R^n)} \:.
\end{equation*}
Hence, the result follows.
\end{proof}

We now turn our attention to the situation when $\mathscr{F}f$ is supported away from the origin. In this case, we need to do some local geometry. For each $n \ge 2$, there exists at least one $n$-dimensional Damek-Ricci space $S$ (in particular, the degenerate case of Real hyperbolic spaces).

We now consider the Riemannian exponential map at $e$, denoted by $\exp_e$ defined on $T_eS$, the tangent space of $S$ at the identity $e$. As Damek-Ricci spaces are connected, simply-connected, complete Riemannian manifolds of non-positive sectional curvature, $\exp_e$ defines a global diffeomorphism from $T_eS$ onto $S$. Now for a general Riemannian manifold, $\exp_e$ is always a local radial isometry. Hence, there exists $\varepsilon >0$, such that $\exp_e$ maps the ball $\mathcal{B}(0,\varepsilon) \subset T_eS$ (with respect to the norm $\tilde{d}(\cdot)$ induced by the Riemannian metric) diffeomorphically onto the geodesic ball $\mathscr{B}(e,\varepsilon) \subset S$, with the property that for any $p \in \mathcal{B}(0,\varepsilon)$, we have
\begin{equation} \label{isometry1}
d(e,exp_e(p))=\tilde{d}(p)\:,
\end{equation}
(see \cite[p. 636]{HRWW}). Moreover, since any two finite dimensional inner product spaces with the same dimension, are isometrically isomorphic, there exists a linear isometry $\mathcal{I}$ from $\R^n$ equipped with the flat metric onto $T_eS$ equipped with the Riemannian metric. Hence for any $x \in \R^n$, we have
\begin{equation} \label{isometry2}
\|x\|=\tilde{d}\left(\mathcal{I}(x)\right)\:.
\end{equation} 
We now restrict to the Euclidean ball $B(o,\varepsilon)$ and consider the composition,
\begin{equation*}
\mathcal{E}:= \exp_e \circ \: \mathcal{I}\:.
\end{equation*}
Thus combining (\ref{isometry1}) and (\ref{isometry2}), it follows that $\mathcal{E}$ maps the ball $B(o,\varepsilon) \subset \R^n$ diffeomorphically onto the geodesic ball $\mathscr{B}(e,\varepsilon) \subset S$, with the property that for any $x \in B(o,\varepsilon)$, we have,
\begin{equation} \label{local_isometry}
d\left(e,\mathcal{E}(x)\right) = \|x\|\:.
\end{equation}
As we are only interested in radial functions, identifying all the points on a geodesic sphere centered at $e$, contained in the geodesic ball $\mathscr{B}(e,\varepsilon)$, solely by their distance from $e$ and using (\ref{local_isometry}), we make the abuse of notation, 
\begin{equation} \label{abuse_of_notation}
\mathcal{E}(x) \sim d\left(e,\mathcal{E}(x)\right)=\|x\|,\:\: \text{ for } x \in B(o,\varepsilon).
\end{equation}
Now under the hypothesis of Theorem \ref{thm2}, a computation similar to Lemma \ref{curve_growth_lemma} using $(\mathscr{H}_3)$ and $(\mathscr{H}_4)$ yields that for all $s \in [0,\delta_2)$ and $t \in [-T,T]$,
\begin{equation*}
\|\gamma_s(t)\| < \frac{3C_6}{2}\delta_2\:.
\end{equation*}
Then furthermore assuming $\delta_2 \le (2\varepsilon)/(3C_6)$, we get that $\|\gamma_s(t)\| < \varepsilon$, for all $s \in [0,\delta_2)$ and $t \in [-T,T]$. Hence it makes sense to consider the pushforwards (by the map $\mathcal{E}$) of the curves $\gamma$ for $s \in [0,\delta_2)$ and $t \in [-T,T]$ , which we write by (\ref{abuse_of_notation}) as,
\begin{equation*}
{\left(\mathcal{E}_* \gamma\right)}_s(t):= \mathcal{E}(\gamma_s(t))\:.
\end{equation*}
We now summarize some important properties of these pushforwarded curves:
\begin{lemma} \label{pushfwd_lemma}
Under the hypothesis of Theorem \ref{thm2}, with $\delta_2 \le (2\varepsilon)/(3C_6)$, where $\varepsilon$ is as in (\ref{local_isometry}), for $s,s' \in [0,\delta_2)$ and $t,t' \in [-T,T]$, we have for $\alpha \in [\frac{1}{2},1]$,
\begin{eqnarray*}
&&(i) \left|d\left(e,{\left(\mathcal{E}_* \gamma\right)}_s(t)\right) - d\left(e,{\left(\mathcal{E}_* \gamma\right)}_s(t')\right)\right| \le  C_4 {|t-t'|}^\alpha \:, \\
&&(ii)\: C_5 |s-s'| \le \left|d\left(e,{\left(\mathcal{E}_* \gamma\right)}_s(t)\right) - d\left(e,{\left(\mathcal{E}_* \gamma\right)}_{s'}(t)\right)\right| \le C_6 |s-s'|\:.
\end{eqnarray*}
\end{lemma}
\begin{proof}
Using (\ref{local_isometry}) and $(\mathscr{H}_3)$, property $(i)$ is proved as follows,
\begin{eqnarray*}
\left|d\left(e,{\left(\mathcal{E}_* \gamma\right)}_s(t)\right) - d\left(e,{\left(\mathcal{E}_* \gamma\right)}_s(t')\right)\right| 
&=& \left|d\left(e,\mathcal{E}(\gamma_s(t))\right)- d\left(e,\mathcal{E}(\gamma_s(t'))\right)\right| \\
&=& \left|\|\gamma_s(t)\|-\|\gamma_s(t')\|\right| \\
& \le & C_4 {|t-t'|}^\alpha \:.
\end{eqnarray*}
Property $(ii)$ can be verified similarly.
\end{proof}

We next discuss some ideas on how to obtain a natural connection between spaces of radial Schwartz class functions on $\R^n$ and $S$ with Spherical Fourier transforms supported away from the origin. These ideas have also been explored recently in \cite{DR} by the author and Ray. For the sake of completion, we briefly mention them here.

By the well-known properties of the Abel transform, we have the following commutative diagram, where
\begin{itemize}
\item $\mathscr{A}_{S,\R}$ is the Abel transform defined from $\mathscr{S}^2(S)_{o}$ to ${\mathscr{S}(\R)}_{e}$\:.
\item $\mathscr{A}_{\R^n,\R}$ is the Abel transform defined from $\mathscr{S}(\R^n)_{o}$ to ${\mathscr{S}(\R)}_{e}$\:.
\item $\wedge$ denotes the Spherical Fourier transform from $\mathscr{S}^2(S)_{o}$ to ${\mathscr{S}(\R)}_{e}$\:.
\item $\mathscr{F}$ denotes the Euclidean Spherical Fourier transform from $\mathscr{S}(\R^n)_{o}$ to ${\mathscr{S}(\R)}_{e}$\:.
\item $\sim$ denotes the 1-dimensional Euclidean Fourier transform from ${\mathscr{S}(\R)}_{e}$ to itself.
\end{itemize}

\[
\begin{tikzcd}[row sep=1.4cm,column sep=1.4cm]
\mathscr{S}^2(S)_{o}\arrow[r,"\mathscr{A}_{S,\R}"] \arrow[dr,swap,"\wedge"] & {\mathscr{S}(\R)}_{e}
\arrow[d,"\sim"] &
\arrow[l,"\mathscr{A}_{\R^n,\R}",swap]  \mathscr{S}(\R^n)_{o} \arrow[dl,"\mathscr{F}"] \\
&  {\mathscr{S}(\R)}_{e}&
\end{tikzcd}
\]

Then, since all the maps above are topological isomorphisms, defining 
\begin{equation*}
\mathscr{A}:=\mathscr{A}^{-1}_{\R^n,\R} \circ \mathscr{A}_{S,\R}\:,
\end{equation*}
one can reduce matters to the following simplified commutative diagram:

\[
\begin{tikzcd}[row sep=1.4cm,column sep=1.4cm]
\mathscr{S}^2(S)_{o}\arrow[r,"\mathscr{A}"] \arrow[dr,swap,"\wedge"] & \mathscr{S}(\R^n)_{o}
\arrow[d,"\mathscr{F}"]   \\
&  {\mathscr{S}(\R)}_{e}&
\end{tikzcd}
\]

Hence, for $g \in \mathscr{S}^2(S)_{o}$, we get
\begin{equation} \label{abel_correspondence}
\widehat{g}(\lambda)= \mathscr{F}(\mathscr{A}g)(\lambda)\:.
\end{equation}  
Next let us define
\begin{equation*}
\mathscr{S}(\R)^\infty_{e} = \left\{\kappa \in \mathscr{S}(\R)_{e} \:|\: 0 \notin Supp(\kappa) \right\} \:.
\end{equation*}
So, $\mathscr{S}(\R)^\infty_{e}$ is the collection of all even Schwartz class functions on $\R$ which are supported outside an interval containing $0$. We now define a map $\mathfrak{m}: \mathscr{S}(\R)^\infty_{e}\to \mathscr{S}(\R)^\infty_{e}$, by the formula
\begin{equation*}
\mathfrak{m}(\kappa)(\lambda):= \frac{{|{\bf c}(\lambda)|}^{-2}}{\lambda^{n-1}} \kappa(\lambda).
\end{equation*}
Because of the derivative estimates of ${|{\bf c}(\lambda)|}^{-2}$ (see (\ref{c-fn_derivative_estimates})), the map $\mathfrak m$ is well defined and is a bijection with the inverse given by,
\begin{equation*}
\mathfrak{m}^{-1}(\kappa)(\lambda):= \frac{\lambda^{n-1}}{{|{\bf c}(\lambda)|}^{-2}} \kappa(\lambda).
\end{equation*}
We also define
\begin{equation*}
\mathscr{S}(\R^n)^\infty_{o}=\{h\in\mathscr{S}(\R^n)_{o}\mid \mathscr{F}h\in \mathscr{S}(\R)^\infty_{e}\},
\end{equation*}
so that $\mathscr{S}(\R^n)^\infty_{o} := \mathscr{F}^{-1}\left(\mathscr{S}(\R)^\infty_{e}\right)$. As the Euclidean Fourier transform $\mathscr{F}$ is a bijection between $\mathscr{S}(\R^n)_{o}$ and $\mathscr{S}(\R)_{e}$, we get an induced map $\mathcal{M}$ obtained by conjugating $\mathfrak{m}$ with $\mathscr{F}$, which is a bijection from $\mathscr{S}(\R^n)^\infty_{o}$ onto itself:

\[
\begin{tikzcd}[row sep=1.4cm,column sep=1.4cm]
\mathscr{S}(\R^n)^\infty_{o} \arrow[r,"\mathcal{M}"] \arrow[d,swap,"\mathscr{F}"] & \mathscr{S}(\R^n)^\infty_{o}
\arrow[d,"\mathscr{F}"]   \\
{\mathscr{S}(\R)}^\infty_{e}  \arrow[r,"\mathfrak{m}"]  &  {\mathscr{S}(\R)}^\infty_{e}&
\end{tikzcd}
\]

Thus for $h \in \mathscr{S}(\R^n)^\infty_{o}$, we get that
\begin{equation*} 
\mathscr{F}(\mathcal{M}h)(\lambda)=\frac{{|{\bf c}(\lambda)|}^{-2}}{\lambda^{n-1}}  \mathscr{F}h(\lambda)\:,
\end{equation*}
and similarly,
\begin{equation} \label{schwartz_multiplier}
\mathscr{F}(\mathcal{M}^{-1}h)(\lambda)=\frac{\lambda^{n-1}}{{|{\bf c}(\lambda)|}^{-2}}  \mathscr{F}h(\lambda)\:.
\end{equation}

\begin{lemma} \label{euclidean_large_freq_lemma}
Under the hypothesis of Theorem \ref{thm2}, with $\delta_2 \le \frac{2}{3C_6}\min\{\varepsilon, R_0\}$, where $\varepsilon$ is as in (\ref{local_isometry}) and $R_0$ is as in Lemma \ref{bessel_series_expansion}, we have for all $f \in \mathscr{S}(\R^n)_o$ with $Supp\left(\mathscr{F}(f)\right) \subset (1,\infty)$,
\begin{equation*} 
{\left(\bigintssss_{\delta_1}^{\delta_2} {\left|\tilde{T}_\gamma f(s)\right|}^2 \:s^{n-1}\:ds\right)}^{1/2} \lesssim \: {\|f\|}_{\dot{H}^{1/4}(\R^n)}\:,
\end{equation*}
where the implicit constant depends only on $\delta_1,\delta_2$ and $n$.
\end{lemma}
\begin{proof}
Let
\begin{equation*}
\mathscr{S}^2(S)^\infty_{o}=\{g\in\mathscr{S}^2(S)_{o}\mid \widehat{g} \in \mathscr{S}(\R)^\infty_{e}\},
\end{equation*}
so that $\mathscr{S}^2(S)^\infty_{o} := \wedge^{-1}\left(\mathscr{S}(\R)^\infty_{e}\right)$. Now as the Spherical Fourier transform $\wedge$ is a bijection from $\mathscr{S}^2(S)_{o}$ onto $\mathscr{S}(\R)_{e}$, its restriction also defines a bijection from $\mathscr{S}^2(S)^\infty_{o}$ onto $\mathscr{S}(\R)^\infty_{e}$. Then combining (\ref{abel_correspondence}), (\ref{schwartz_multiplier}) and the fact that $Supp\left(\mathscr{F}(f)\right) \subset (1,\infty)$, we define 
\begin{equation*}
g:= \left(\mathscr{A}^{-1} \circ \mathcal{M}^{-1} \right)f\:.
\end{equation*}
The interaction of $g$ with the Fourier transform is best understood via the following commutative diagram, where each arrow is a bijection: 
\[
\begin{tikzcd}[row sep=1.4cm,column sep=1.4cm]
\mathscr{S}(\R^n)^\infty_{o} \arrow[r,"\mathcal{M}^{-1}"] \arrow[d,swap,"\mathscr{F}"] & \mathscr{S}(\R^n)^\infty_{o}
\arrow[r,"\mathscr{A}^{-1}"]
\arrow[d,swap,"\mathscr{F}"] & \mathscr{S}^2(S)^\infty_{o} \arrow[dl,"\wedge"]  \\
{\mathscr{S}(\R)}^\infty_{e}  \arrow[r,"\mathfrak{m}^{-1}"]  &  {\mathscr{S}(\R)}^\infty_{e}&
\end{tikzcd}
\]

Thus $g \in \mathscr{S}^2(S)^\infty_{o}$ with \begin{equation} \label{euclidean_large_freq_lemma_eq1}
\widehat{g}(\lambda)=\mathscr{F}\left(\mathcal{M}^{-1} f\right)(\lambda)=\frac{\lambda^{n-1}}{{|{\bf c}(\lambda)|}^{-2}}  \mathscr{F}f(\lambda)\:.
\end{equation}
So in particular, $Supp(\widehat{g}) \subset (1,\infty)$.

\medskip

We also note that using the hypothesis, a computation similar to Lemma \ref{curve_growth_lemma} implies that $\|\gamma_s(t)\| < \min\{\varepsilon,R_0\}$, for all $s \in [0,\delta_2)$ and $t \in [-T,T]$. Hence by (\ref{local_isometry}), we see that
\begin{equation} \label{euclidean_large_freq_lemma_eq2}
d\left(e,{\left(\mathcal{E}_* \gamma\right)}_s(t)\right) = \|\gamma_s(t)\| < \min\{\varepsilon,R_0\}\:.
\end{equation} 
 
Then for $s \in (\delta_1,\delta_2)$, using Lemmata \ref{pushfwd_lemma} and \ref{curve_growth_lemma}, the local growth asymptotics of the density function of $S$ (\ref{density_function}), the relationship between $\widehat{g}$ and $\mathscr{F}f$ (\ref{euclidean_large_freq_lemma_eq1}) and the isometric relation (\ref{euclidean_large_freq_lemma_eq2}), we obtain
\begin{eqnarray*}
\left|T_{\mathcal{E}_* \gamma,4}\:g(s)\right| 
& \asymp & \left|\int_1^\infty \J_{\frac{n-2}{2}}\left(\lambda\: d\left(e,{\left(\mathcal{E}_* \gamma\right)}_s(t(s))\right)\right)\:e^{it(s)\lambda^2}\:\widehat{g}(\lambda)\: {|{\bf c}(\lambda)|}^{-2}\: d\lambda \right| \\
&=& \left|\int_1^\infty \J_{\frac{n-2}{2}}\left(\lambda\: \|\gamma_s(t(s))\|\right)\:e^{it(s)\lambda^2} \mathscr{F}f(\lambda)\: \lambda^{n-1}\: d\lambda \right| \\
&=& \left|\tilde{T}_\gamma f(s)\right|\:.
\end{eqnarray*}
Thus again by the local growth asymptotics of the density function of $S$, it follows that
\begin{equation} \label{euclidean_large_freq_lemma_eq3}
{\left(\bigintssss_{\delta_1}^{\delta_2} {\left| T_{\mathcal{E}_* \gamma,4}\: g(s)\right|}^2 A(s)\:ds\right)}^{1/2} \asymp {\left(\bigintssss_{\delta_1}^{\delta_2} {\left|\tilde{T}_\gamma f(s)\right|}^2 \:s^{n-1}\:ds\right)}^{1/2}\:.
\end{equation}
Using (\ref{euclidean_large_freq_lemma_eq1}) and the large frequency asymptotics of ${|{\bf c}(\lambda)|}^{-2}$ (\ref{plancherel_measure}), we also compare the Homogeneous Sobolev norms of $g$ and $f$:
\begin{eqnarray} \label{euclidean_large_freq_lemma_eq4}
{\|g\|}_{\dot{H}^{1/4}(S)} &=& {\left(\int_1^\infty \lambda^{1/2}\:{\left|\widehat{g}(\lambda)\right|}^2\: {|{\bf c}(\lambda)|}^{-2}\: d\lambda\right)}^{1/2} \nonumber\\
& = & {\left(\int_1^\infty \lambda^{1/2}\:{\left|\mathscr{F}f(\lambda)\right|}^2\:\frac{\lambda^{2(n-1)}}{{|{\bf c}(\lambda)|}^{-4}}\: {|{\bf c}(\lambda)|}^{-2}\: d\lambda\right)}^{1/2} \nonumber\\
& \asymp & {\left(\int_1^\infty \lambda^{1/2}\:{\left|\mathscr{F}f(\lambda)\right|}^2\:\lambda^{n-1}\: d\lambda\right)}^{1/2} \nonumber\\
&=& {\|f\|}_{\dot{H}^{1/4}(\R^n)}\:.
\end{eqnarray}
Then by Lemma \ref{pushfwd_lemma}, Corollary \ref{small_annuli_cor}, (\ref{euclidean_large_freq_lemma_eq3}) and (\ref{euclidean_large_freq_lemma_eq4}), the result follows.
\end{proof}
We now complete the proof of Theorem \ref{thm2}.
\begin{proof}[Proof of Theorem \ref{thm2}]
Let $f \in \mathscr{S}(\R^n)_{o}$. Then $\mathscr{F}f \in \mathscr{S}(\R)_{e}$. Now let us choose an auxiliary non-negative even function $\psi \in C^\infty_c(\R)$ such that $Supp(\psi) \subset \{\xi: \frac{1}{2} < |\xi| < 2\}$ and 
\begin{equation*}
\displaystyle\sum_{k=-\infty}^\infty \psi(2^{-k} \xi)=1\:,\: \xi \ne 0\:.
\end{equation*} 
We set
\begin{equation*}
\psi_1(\xi):= \displaystyle\sum_{k=-\infty}^0 \psi(2^{-k} \xi) \:,\: \text{ and } \psi_2(\xi):= \displaystyle\sum_{k=1}^\infty \psi(2^{-k} \xi) \:. 
\end{equation*}
Then both $\psi_1$ and $\psi_2$ are even non-negative $C^\infty$ functions with $Supp(\psi_1) \subset (-2,2)$, $Supp(\psi_2) \subset \R \setminus (-1,1)$ and $\psi_1+\psi_2 \equiv 1$. Then we consider,
\begin{equation*}
\mathscr{F}f = \left(\mathscr{F}f\right) \psi_1 + \left(\mathscr{F}f \right) \psi_2\:.
\end{equation*} 
As $\left(\mathscr{F}f\right) \psi_j$ for $j=1,2$, are even Schwartz class functions on $\R$, by the Schwartz isomorphism theorem, there exist $f_1,f_2 \in \mathscr{S}(\R^n)_{o}$ such that $\mathscr{F}f_j=\left(\mathscr{F}f\right) \psi_j$ for $j=1,2$. Then by Lemmata \ref{euclidean_small_freq_lemma} and \ref{euclidean_large_freq_lemma}, we get the desired estimate on the linearized maximal function,
\begin{eqnarray*}
&&{\left(\bigintssss_{\delta_1}^{\delta_2} {\left|\tilde{T}_\gamma f(s)\right|}^2 \:s^{n-1}\:ds\right)}^{1/2} \\
&\le & {\left(\bigintssss_{\delta_1}^{\delta_2} {\left|\tilde{T}_\gamma f_1(s)\right|}^2 \:s^{n-1}\:ds\right)}^{1/2} + {\left(\bigintssss_{\delta_1}^{\delta_2} {\left|\tilde{T}_\gamma f_2(s)\right|}^2 \:s^{n-1}\:ds\right)}^{1/2} \\
& \lesssim & {\left(\int_0^2 \lambda^{1/2}\:{\left|\mathscr{F}f_1(\lambda)\right|}^2\:\lambda^{n-1}\: d\lambda\right)}^{1/2} + {\left(\int_1^\infty \lambda^{1/2}\:{\left|\mathscr{F}f_2(\lambda)\right|}^2\:\lambda^{n-1}\: d\lambda\right)}^{1/2} \\
& \le & 2 \:{\|f\|}_{\dot{H}^{1/4}(\R^n)}\:.
\end{eqnarray*}
\end{proof}

\begin{proof}[Proof of Theorem \ref{thm3}]
By arguments mentioned in the Introduction, in order to prove the result on almost everywhere pointwise convergence, it suffices to prove that given any $0<r_1<r_2<\infty$, there exists some fixed $T>0$, such that one has for all $f \in \mathscr{S}(\R^n)_o$, the maximal estimate,
\begin{equation} \label{thm3_pf_eq1}
{\left(\bigintssss_{r_1}^{r_2} {\left(\displaystyle\sup_{t \in [-T,T]}\left|\tilde{S}_t f(\|\gamma_s(t)\|)\right|\right)}^2 \:s^{n-1}\:ds\right)}^{1/2} \lesssim \: {\|f\|}_{\dot{H}^{1/4}(\R^n)}\:,
\end{equation}
where the implicit constant can possibly depend only on $r_1,r_2$ and $n$. We first note that by compactness of $[0,r_2]$, there exists  $0<T< {\left(r_1/2C_7\right)}^2$, such that
\begin{equation*}
\|\gamma_s(t)\|=s+C_7t^{1/2}\:,
\end{equation*}
for all $s \in [0,r_2)$ and $t \in [-T,T]$. Then for all $s \in [0,r_2)$ and $t,t' \in [-T,T]$,
\begin{equation*}
\left|\|\gamma_s(t)\|-\|\gamma_s(t')\|\right| = C_7 |t^{1/2}- {t'}^{1/2}| \le C_7 {|t- {t'}|}^{1/2} \:.
\end{equation*} 
Also for all $s,s' \in [0,r_2)$ and $t \in [-T,T]$,
\begin{equation*}
\left|\|\gamma_s(t)\|-\|\gamma_{s'}(t)\|\right| = |s- s'| \:.
\end{equation*}
Therefore, $\gamma$ satisfies $(\mathscr{H}_3)$ for $\alpha=1/2$ and $(\mathscr{H}_4)$ for $s,s' \in [0,r_2)$ and $t,t' \in [-T,T]$, for some fixed $0<T< {\left(r_1/2C_7\right)}^2$.

Let $\delta \in (0,r_2)$ be sufficiently small and set $\eta=r_2/\delta$ and $\delta'=r_1/\eta$. For $f \in \mathscr{S}(\R^n)_o$, we consider 
\begin{equation*}
f_\eta(x):= f(\eta x)\:.
\end{equation*}
Now the solutions of the Schr\"odinger equation with  initial data $f$ and $f_\eta$ are related as follows: 
\begin{eqnarray*}
\tilde{S}_t f(\|\gamma_s(t)\|) &=&\int_0^\infty \J_{\frac{n-2}{2}}\left(\lambda \|\gamma_s(t)\|\right) e^{it\lambda^2}\:\mathscr{F}f(\lambda)\:\lambda^{n-1}\:d\lambda \\
&=& \eta^n \int_0^\infty \J_{\frac{n-2}{2}}\left(\lambda \|\gamma_s(t)\|\right) e^{it\lambda^2}\:\mathscr{F} f_\eta(\eta \lambda)\:\lambda^{n-1}\:d\lambda \\
&=& \int_0^\infty \J_{\frac{n-2}{2}}\left(\lambda \frac{\|\gamma_s(t)\|}{\eta}\right) e^{i\left(\frac{t}{\eta^2}\right)\lambda^2}\:\mathscr{F}f_\eta(\lambda)\:\lambda^{n-1}\:d\lambda \:.
\end{eqnarray*}
Now for $s \in [0,r_2)$ and $t \in [-T,T]$, we have 
\begin{equation*}
\frac{\|\gamma_s(t)\|}{\eta} = \frac{s}{\eta}+C_7 \frac{t^{1/2}}{\eta}= \frac{s}{\eta}+C_7 {\left(\frac{t}{\eta^2}\right)}^{1/2} = \|\gamma_{s/\eta}(t/\eta^2)\|\:.
\end{equation*}
Therefore, for $s \in [0,r_2)$ and $t \in [-T,T]$,
\begin{eqnarray*}
\tilde{S}_t f(\|\gamma_s(t)\|) &=& \int_0^\infty \J_{\frac{n-2}{2}}\left(\lambda \|\gamma_{s/\eta}(t/\eta^2)\|\right) e^{i\left(\frac{t}{\eta^2}\right)\lambda^2}\:\mathscr{F}f_\eta(\lambda)\:\lambda^{n-1}\:d\lambda \\
&=& \tilde{S}_{t/\eta^2} f_\eta(\|\gamma_{s/\eta}(t/\eta^2)\|)\:.
\end{eqnarray*}
Then as $T< {\left(r_1/2C_7\right)}^2$, we have 
\begin{equation*}
{\left(\frac{T}{\eta^2}\right)}^{1/2} < \frac{\delta'}{2C_7}\:,
\end{equation*}
and hence by Theorem \ref{thm2},
\begin{eqnarray*}
&&{\left(\bigintssss_{r_1}^{r_2} {\left(\displaystyle\sup_{t \in [-T,T]}\left|\tilde{S}_t f(\|\gamma_s(t)\|)\right|\right)}^2 \:s^{n-1}\:ds\right)}^{1/2} \\
&=& {\left(\bigintssss_{r_1}^{r_2} {\left(\displaystyle\sup_{t \in [-T,T]}\left|\tilde{S}_{t/\eta^2} f_\eta(\|\gamma_{s/\eta}(t/\eta^2)\|)\right|\right)}^2 \:s^{n-1}\:ds\right)}^{1/2} \\
&=& {\left(\bigintssss_{r_1}^{r_2} {\left(\displaystyle\sup_{t \in \left[-\frac{T}{\eta^2},\frac{T}{\eta^2}\right]}\left|\tilde{S}_{t} f_\eta(\|\gamma_{s/\eta}(t)\|)\right|\right)}^2 \:s^{n-1}\:ds\right)}^{1/2} \\
&=& \eta^{n/2}\: {\left(\bigintssss_{\delta'}^{\delta} {\left(\displaystyle\sup_{t \in \left[-\frac{T}{\eta^2},\frac{T}{\eta^2}\right]}\left|\tilde{S}_{t} f_\eta(\|\gamma_{s}(t)\|)\right|\right)}^2 \:s^{n-1}\:ds\right)}^{1/2} \\
& \lesssim &  \eta^{n/2}\:{\|f_\eta\|}_{\dot{H}^{1/4}(\R^n)}\:.
\end{eqnarray*}
We next note that
\begin{eqnarray*}
{\|f_\eta\|}_{\dot{H}^{1/4}(\R^n)} &=& {\left(\int_0^\infty \lambda^{1/2}\:{\left|\mathscr{F} f_\eta(\lambda)\right|}^2 \lambda^{n-1}\:d\lambda\right)}^{1/2} \\
&=& \eta^{\left(\frac{1}{4}-\frac{n}{2}\right)} {\left(\int_0^\infty \lambda^{1/2}\:{\left|\mathscr{F}f(\lambda)\right|}^2 \lambda^{n-1}\:d\lambda\right)}^{1/2} \\
&=& \eta^{\left(\frac{1}{4}-\frac{n}{2}\right)} {\|f\|}_{\dot{H}^{1/4}(\R^n)}\:, 
\end{eqnarray*}
which we plug in the above to obtain
\begin{eqnarray*}
{\left(\bigintssss_{r_1}^{r_2} {\left(\displaystyle\sup_{t \in [-T,T]}\left|\tilde{S}_t f(\|\gamma_s(t)\|)\right|\right)}^2 \:s^{n-1}\:ds\right)}^{1/2} & \lesssim & \eta^{\frac{1}{4}} {\|f\|}_{\dot{H}^{1/4}(\R^n)}\:.
\end{eqnarray*}
Thus we obtain (\ref{thm3_pf_eq1}), which completes the proof.
\end{proof}

\section{Concluding remarks}
In this section, we make some remarks and pose some new problems:
\begin{enumerate}
\item On $\mathbb{H}^3(-1) \cong SL(2,\C)/SU(2)$, for $\beta >3/2$, by Sobolev imbedding \cite[Theorem 4.7]{CGM} and standard arguments, the non-tangential convergence to the initial data holds almost everywhere. Whereas by our counter-example (Theorem \ref{counter-example}), the non-tangential convergence to the initial data fails on a set of positive ($SL(2,\C)$-invariant) Riemannian measure if $\beta \le 1/2$. This gives rise to the natural question of what happens when $\beta \in (1/2,3/2]$. The analogue of Theorem \ref{counter-example} and the above question in the great generality of Damek-Ricci spaces turns out to be even more interesting.

\medskip

\item In this article, the two Euclidean results (Theorems \ref{thm2} and \ref{thm3}) only give partial analogues of Theorem \ref{thm1} for $\R^n$. Thus seeking a complete analogue of Theorem \ref{thm1} for $\R^n$ is a natural question.

\end{enumerate}

\section*{Acknowledgements} 
The author would like to thank Prof. Swagato K. Ray for comments and suggestions. The author is supported by a research fellowship of Indian Statistical Institute.

\section*{Statements and Declarations}
{\bf Conflict of interest:} The author declares no conflict of interest related to this work.

\bibliographystyle{amsplain}

\end{document}